\documentclass[11pt]{amsart}

\usepackage[draft]{changes}

\definechangesauthor[color=red]{pe}
\definechangesauthor[color=red]{xm}
\definechangesauthor[color=red]{hx}
\newcommand{\stkout}[1]{\ifmmode\text{\sout{\ensuremath{#1}}}\else\sout{#1}\fi}
\setdeletedmarkup{\stkout{#1}}
\colorlet{Changes@Color}{red}

\usepackage{amssymb}

\usepackage{graphics}
\usepackage{graphicx}

\usepackage{amsmath}
\usepackage{amsthm}
\usepackage{amsfonts}
\usepackage{mathtools}
\usepackage{xcolor}
\usepackage[english]{babel}
\usepackage[margin=1.0in]{geometry}
\usepackage[colorlinks=true,
        linkcolor=blue]{hyperref}
\usepackage{bbm}
\usepackage{verbatim}
\usepackage{extarrows}
\usepackage{blkarray}
\usepackage{tensor}

\usepackage[T1]{fontenc}

\numberwithin{equation}{section}

\newtheorem{prop}{Proposition}
\newtheorem{lemma}[prop]{Lemma}

\newtheorem{thm}[prop]{Theorem}
\newtheorem*{thm*}{Theorem}
\newtheorem{cor}[prop]{Corollary}

\numberwithin{prop}{section}

\newtheorem{defn}[prop]{Definition}
\newtheorem*{defn*}{Definition}
\theoremstyle{definition}
\newtheorem{ex}[prop]{Example}
\newtheorem*{ex*}{Example}
\newtheorem{rmk}[prop]{Remark}

\definecolor{c1}{rgb}{0.2,0.4,0.5}
\definecolor{c2}{rgb}{0.1,0.3,0.5}
\definecolor{c3}{rgb}{0.2,0.7,0.5}
\usepackage{tikz}

\def \k {K\"ahler }
\def \ke {K\"ahler--Einstein }
\newcommand{\oo}[1]{\overline{#1}}

\newcommand{\bC}{\mathbb{C}}

\newcommand{\dbar}{\oo\partial}

\def\Ol{\overline}

\DeclareFontFamily{U}{MnSymbolC}{}
\DeclareSymbolFont{MnSyC}{U}{MnSymbolC}{m}{n}
\DeclareFontShape{U}{MnSymbolC}{m}{n}{
	<-6>  MnSymbolC5
	<6-7>  MnSymbolC6
	<7-8>  MnSymbolC7
	<8-9>  MnSymbolC8
	<9-10> MnSymbolC9
	<10-12> MnSymbolC10
	<12->   MnSymbolC12}{}
\DeclareMathSymbol{\intprod}{\mathbin}{MnSyC}{'270}

\DeclareMathOperator{\Ric}{Ric}
\DeclareMathOperator{\End}{End}

\DeclareMathOperator{\rank}{Rank}

\DeclareMathOperator{\Vol}{Vol}

\DeclareMathOperator{\Gl}{Gl}

\DeclareMathOperator{\Iso}{Iso}

\begin{document}

\title[]{K\"ahler--Einstein metrics and obstruction flatness II: Unit Sphere bundles}

\begin{abstract} 
This paper concerns obstruction flatness of hypersurfaces $\Sigma$ that arise as unit sphere bundles $S(E)$ of Griffiths negative Hermitian vector bundles $(E, h)$ over \k manifolds $(M, g).$ We prove that if the curvature of $(E, h)$  satisfies a splitting condition and $(M,g)$ has constant Ricci eigenvalues, then $S(E)$ is obstruction flat. If, in addition, all these eigenvalues are strictly less than one and $(M,g)$ is complete, then we show that the corresponding ball bundle admits a complete K\"ahler-Einstein metric.

\end{abstract}



\author [Ebenfelt]{Peter Ebenfelt}
\address{Department of Mathematics, University of California at San Diego, La Jolla, CA 92093, USA} \email{{pebenfelt@ucsd.edu}}

\author[Xiao]{Ming Xiao}
\address{Department of Mathematics, University of California at San Diego, La Jolla, CA 92093, USA}
\email{{m3xiao@ucsd.edu}}

\author [Xu]{Hang Xu}
\address{School of Mathematics (Zhuhai), Sun Yat-sen University, Zhuhai, Guangdong 519082, China}
\email{{xuhang9@mail.sysu.edu.cn}}

\thanks{The first author was supported in part by the NSF grant DMS-1900955 and DMS-2154368. The second author was supported in part by the NSF grants DMS-1800549 and DMS-2045104. The third author was supported in part by the NSFC grant No. 12201040.}

\maketitle

\section{Introduction}

In this paper, the authors continue their investigation of K\"ahler--Einstein metrics and obstruction flatness in the context of domains in vector bundles. In a recent paper \cite{EXX22} the authors studied obstruction flatness of CR hypersurfaces that arise as the unit circle bundle of a negative Hermitian line bundle over a \k manifold. The authors proved, among other results, that for a negative line bundle $(L, h)$ over a complex manifold $M$, if the \k metric $g$ induced by $-\Ric(L, h)$ has constant Ricci eigenvalues, then the unit circle bundle $S(L)$ is obstruction flat. If, in addition, all the Ricci eigenvalues are strictly less than $1$, then the disk bundle admits a complete \ke metric. It is natural to consider also the more general case of unit sphere bundles in Hermitian vector bundles of higher rank. The goal of this paper is to find the right conditions on a vector bundle that will guarantee obstruction flatness of the corresponding sphere bundle. This turns out to be more subtle than the line bundle case (cf., e.g., Remark \ref{rem:Webster} below and the penultimate paragraph of this introduction) and we need to pose an additional condition on the curvature, beyond the conditions on
the Ricci curvature that we pose in the line bundle case, as will be explained below.

We begin by introducing the notion of obstruction flatness in its classical context. On a smoothly bounded strongly pseudoconvex domain $\Omega\subset \mathbb{C}^n$, $n\geq 2$, the existence of a complete \ke metric on $\Omega$ is governed by the following Dirichlet problem for Fefferman's complex Monge-Amp\`ere equation:
\begin{equation}\label{Dirichlet problem}
	\begin{dcases}
		J(u):=(-1)^n \det \begin{pmatrix}
			u & u_{\oo{z_k}}\\
			u_{z_j} & u_{z_j\oo{z_k}} \\
		\end{pmatrix}=1 & \mbox{in } \Omega\\
		u=0 & \mbox{on } \partial\Omega
	\end{dcases}
\end{equation}
with $u>0$ in $\Omega$. If $u$ is a solution of \eqref{Dirichlet problem}, then $-\log u$ is the \k potential of a complete \ke metric on $\Omega$ with negative Ricci curvature.
Fefferman \cite{Fe2} established the existence of an approximate solution $\rho\in C^{\infty}(\oo{\Omega})$ of \eqref{Dirichlet problem} that only satisfies $J(\rho)=1+O(\rho^{n+1})$, and showed that such a $\rho$ is unique modulo $O(\rho^{n+2})$. Such an approximate solution $\rho$ is often referred to as a {\em Fefferman defining function}. Cheng and Yau \cite{CheYau} then proved the existence and uniqueness of an exact solution $u\in C^{\infty}(\Omega)$ to \eqref{Dirichlet problem}, now called the Cheng--Yau solution. Lee and Melrose \cite{LeeMel82} showed that the Cheng--Yau solution has the following asymptotic expansion:
\begin{equation}\label{Lee-Melrose expansion}
	u\sim \rho\sum_{k=0}^{\infty} \eta_k\bigl(\rho^{n+1}\log\rho\bigr)^k,
\end{equation}
where each $\eta_k\in C^{\infty}(\oo{\Omega})$ and $\rho$ is a Fefferman defining function.
The expansion \eqref{Lee-Melrose expansion} shows that, in general, the Cheng--Yau solution $u$ can only be expected to have a finite degree of boundary smoothness; namely, $u\in C^{n+2-\varepsilon}(\oo{\Omega})$ for any $\varepsilon>0$. Graham \cite{Graham1987a} showed that the obstruction to $C^{\infty}$ boundary regularity of the Cheng--Yau solution is in fact given by the lowest order obstruction $\eta_1|_{\partial\Omega}$, the restriction of $\eta_1$ to the boundary. More precisely, in \cite{Graham1987a} Graham proved that if $\eta_1|_{\partial\Omega}$ vanishes identically (on $\partial\Omega$), then every $\eta_k$ vanishes to infinite order on $\partial\Omega$ for all $k\geq 1$. For this reason, $\eta_1|_{\partial\Omega}$ is called the \emph{obstruction function}.
Graham also showed ({\em op.\ cit.}) that, for any $k\geq 1$, the coefficients $\eta_k$ mod $O(\rho^{n+1})$ are independent of the choice of Fefferman defining function $\rho$ and are locally determined by the local CR geometry of $\partial\Omega$.
As a consequence, the $\eta_k$ mod $O(\rho^{n+1})$, for $k\geq 1$, are local CR invariants that can be defined on any strongly pseudoconvex CR hypersurface in a complex manifold. In particular, the obstruction function $\mathcal{O}=\eta_1|_{\partial\Omega}$ is a local CR invariant that can be defined on any strongly pseudoconvex CR hypersurface $\Sigma$.
If $\Sigma$ is a CR hypersurface for which the obstruction function $\mathcal{O}$ vanishes identically, then $\Sigma$ is said to be \emph{obstruction flat}. The most basic examples of obstruction flat hypersurfaces are the sphere $\{z\in \mathbb{C}^n: |z|=1 \}$ and, more generally, any CR hypersurface that is locally CR diffeomorphic to an open piece of the sphere; a CR hypersurface that is locally CR diffeomorphic to an open piece of the sphere in a neighborhood of any point is called {\em spherical}.

As mentioned above, the main aim of this paper is to extend the authors' results in \cite{EXX22} concerning obstruction flatness of unit circle bundles in negative Hermitian line bundles over \k manifolds to the more general situation of sphere bundles in Hermitian vector bundles of higher rank. To formulate our results, we first review some standard facts and notions concerning the geometry of Hermitian vector bundles and \k manifolds. Let $(E, h)$ be a Hermitian (holomorphic) vector bundle over a complex manifold $M$. Denote by $\pi: E\rightarrow M$ the canonical projection and by $E_z=\pi^{-1}(z)$ the fiber at $z$. Let $\Theta=\Theta_{E, h}$ be the associated curvature form of the Chern connection of $(E, h)$; thus, $\Theta$ is an $\End(E)$-valued $(1,1)$-form. At each point $z\in M$, the tensor $\Theta_z=\Theta(z)$ can be regarded as a Hermitian bilinear form on $E_z\otimes T^{1,0}_z$. We make the following definition, which will be used in the main results.

\begin{defn}\label{def curvature split} {\rm
	The curvature $\Theta$ {\em splits} if, at every $z\in M$, there exists a Hermitian form $H_z$ on $T_z^{1,0}M$ such that $\Theta_z=h\cdot H_z$; or, equivalently,  for any $e\in E_z$ and any $v \in T_z^{1,0}M$,
	\begin{equation}\label{curvature splits}
		\Theta_z(e\otimes v, e\otimes v)=h(e,e) \cdot H_z(v, v).
	\end{equation}	
}
\end{defn}

In addition, we say that the vector bundle $(E, h)$ is {\em curvature split} if its curvature $\Theta$ splits. We remark that when this is the case, $H_z$ is equal to the Ricci curvature of $(E, h)$ at $z$ up to a scaling factor. See the paragraph before \eqref{H is the Ricci curvature} in $\S$\ref{Sec preliminaries}.

When the Hermitian vector bundle $(E,h)$ is Griffiths negative, i.e., $\sqrt{-1}\Theta_z(e\otimes v, e\otimes v)<0$ for all non-zero $e\in E_z$, $v\in T_z^{1,0}M$ and $z\in M$,  the negative of its Ricci, $-\Ric(E, h)$,  induces a K\"ahler metric $g$ on $M$. By Lemma 1.2.2 in Mok-Ng \cite{MoNg}, the corresponding sphere bundle $S(E)=\{e \in E: |e|_h=1 \}$ is strongly pseudoconvex; here, $|e|_h=\sqrt{h(e,e)}$ denotes the norm of $e$ with respect to the metric $h$.

Given an $n$-dimensional \k manifold $(M, g)$, let $\mathrm{Ric}=-i\partial \Ol{\partial} \log \det(g)$ denote the associated Ricci tensor.  The latter naturally induces an endomorphism, the {\em Ricci endomorphism}, of the holomorphic tangent space $T_z^{1,0}M$ given by $\mathrm{Ric} \cdot g^{-1}$ for $z\in M$.
The eigenvalues of this endomorphism will be referred to as
the {\em Ricci eigenvalues} of $(M,g)$ and, by design, are functions of $z\in M$. 
All Ricci eigenvalues are real-valued as both  $\mathrm{Ric}$ and $g$ are Hermitian tensors.
For a fixed $z$, we label the Ricci eigenvalues such that $\lambda_1(z) \leq \cdots \leq \lambda_n (z)$.
Note that the sum of the $\lambda_i(z)$, i.e., the trace of the Ricci endomorphism, gives the scalar curvature at $z$. 
The \k manifold $(M, g)$ is said to have {\em constant Ricci eigenvalues}, if each $\lambda_i(z)$, for $1 \leq i \leq n$, is a constant function on $M$; equivalently, the characteristic polynomial of the Ricci endomorphism, $\mathrm{Ric} \cdot g^{-1}: T_z^{1,0}M \rightarrow T_z^{1,0}M$ is the same at every point $z \in M$.

Our first main result is as follows.
\begin{thm}\label{thm obstruction flat}
	Let $(E, h)$ be a Hermitian holomorphic vector bundle over a complex manifold $M$. Suppose $(E, h)$ is Griffiths negative and its curvature splits. Let $g$ be the \k metric on $M$ induced by $-\Ric(E, h)$. If $(M, g)$ has constant Ricci eigenvalues, then $S(E)$ is obstruction flat.
\end{thm}


Recall that the notion of obstruction flatness originates in the context of complete \ke metrics on domains. It is natural to ask whether a complete \ke metric exists (globally) on the corresponding ball bundle $B(E)=\{e \in E: |e|_h <1 \}$ of $(E, h)$ in the situations we are considering. We shall prove the following result:
\begin{thm}\label{thm ke metric}
	Let $n\geq 1$, $k\geq 1$ and $m=n+k$. Let $M$ be a complex manifold of dimension $n$ and $(E, h)$ a Hermitian holomorphic vector bundle over $M$ of rank $k$ such that $(E, h)$ is Griffiths negative and its curvature splits. Let $g$ be the \k metric on $M$ induced by $-\Ric(E, h)$. Assume $(M, g)$ is complete and has constant Ricci eigenvalues. If all the Ricci eigenvalues are strictly less than $1$, then the ball bundle $B(E)$ admits a unique complete \ke metric $\widetilde{g}$ with Ricci curvature equal to $-(m+1)$. Moreover, this metric is induced by the following \k form
	\begin{equation}\label{thm ke metric eq}
		\widetilde{\omega}(w, \oo{w}):=-\frac{1}{m+1}\pi^*(\Ric)|_w+\frac{1}{m+1}\pi^*(\omega)|_w-i\partial\dbar\log\phi(|w|_h),
	\end{equation}
	where $\omega$ and $\Ric$ are respectively the \k and the Ricci form of $(M,g)$, $\pi: E \rightarrow M$ the canonical fiber projection of the vector bundle. Moreover, $\phi: (-1, 1) \rightarrow \mathbb{R}^+$ is an even real analytic function that depends only on the characteristic polynomial of the Ricci endomorphism of $(M, g)$. (More precisely, $\phi$ is given by Proposition $\ref{prpnzh2}$ by choosing $\lambda_i$'s to be the Ricci eigenvalues).
\end{thm}


If the Hermitian vector bundle comes from a direct sum of copies of a single Hermitian line bundle, then it is automatically curvature split (cf. Proposition \ref{curvature split when the bundle splits prop}). Thus we have the following corollary of Theorems \ref{thm obstruction flat} and \ref{thm ke metric}.
\begin{cor}\label{cor ke metric on direct sum}
	Let $(L, h_0)$ be a negative line bundle over a complex manifold $M$. Set
	\begin{equation*}
		(E, h)=(L, h_0)\oplus \cdots \oplus (L, h_0),
	\end{equation*}
	where there are $k$ copies of $(L, h_0)$ on the right hand side.
	Let $g$ be the \k metric induced by $-\Ric(E, h)$. If $(M, g)$ has constant Ricci eigenvalues, then $S(E)$ is obstruction flat. Furthermore, if in addition $(M, g)$ is complete and all the Ricci eigenvalues are strictly less than $1$, then the ball bundle $B(E)$ admits a unique complete \ke metric with Ricci curvature equal to $-(m+1)$, where $m=n+k$ and $n$ is the dimension of $M$.
\end{cor}

\begin{rmk}\label{rem:Webster}
	In the setting of Corollary \ref{cor ke metric on direct sum},  Webster \cite{Web02} proved that for $n=1$ and $k\geq 2$, $S(E)$ is spherical if and only if $(M, g)$ has constant Gauss curvature $K=-2/k$. Note that the metric $g$ here is $k$ multiple of the metric used in \cite{Web02}. This result illustrates the difference between the case where $(E,h)$ is a line bundle ($k=1$) and the case where it is a vector bundle of rank $k\geq 2$. In the former case provided $M$ is compact, the circle bundle is spherical if and only if $K$ is constant, regardless of its value.
\end{rmk}

\begin{rmk}
Combining Webster's result in Remark \ref{rem:Webster} with Corollary \ref{cor ke metric on direct sum} yields that, for $n=1$ and $k\geq 2$, if the Gauss curvature of $(M, g)$ is constant but not equal to $-2$, then $S(E)$ is obstruction flat but not spherical. For example, taking $M=\mathbb{CP}^1$ and $(L, h_0)$ as the tautological line bundle in Corollary \ref{cor ke metric on direct sum}, we find that $S(E)$ is a compact, obstruction flat and non-spherical CR hypersurface for $k\geq 2$. (More examples of obstruction flat CR hypersurface are provided in $\S$\ref{Sec an interesting case}.)
\end{rmk}

\begin{rmk}
	The condition for existence of a complete K\"ahler--Einstein metric in Theorem \ref{thm ke metric} and Corollary \ref{cor ke metric on direct sum} is optimal, in the sense that the conclusion fails if some Ricci eigenvalue is greater than or equal to $1$. Indeed, the following statement follows from van Coevering's work in \cite[Theorem 1.1 and Corollary 1.3]{Coe12}.	{\em Let $M$ be a compact complex manifold of dimension $n$ and $(E, h)$ a Hermitian holomorphic vector bundle over $M$ of rank $k$. Suppose $(E, h)$ is Griffiths negative. Let $g$ be the \k metric on $M$ induced by $-\Ric(E, h)$. Assume $(M, g)$ has constant Ricci eigenvalues. Then all these Ricci eigenvalues are strictly less than $1$ if and only if the ball bundle $B(E)$ admits a unique complete \ke metric $\widetilde{g}$ with Ricci curvature equal to $-(n+k+1)$.} We point out, however, that the conditions in van Coevering's work \cite{Coe12} are formulated in terms of negativity of Chern classes: $-c_1(M)-c_1(E)>0$. In the context of constant Ricci eigenvalues, it can be shown that this condition is equivalent to all Ricci eigenvalues being $<1$, as in Theorem \ref{thm ke metric} and Corollary \ref{cor ke metric on direct sum}. To see the equivalence, one can verify that when $-c_1(M)-c_1(E)>0$, we have 
\begin{equation}\label{eqvc}
	\int_M S_k(1-\lambda_1,\ldots,1-\lambda_n)\,\omega^n>0, \quad k=1,\ldots, n,
\end{equation}
where $S_k$ denotes the symmetric polynomial of degree $k$ and $\lambda_j$ the Ricci eigenvalues.
In the case where the Ricci eigenvalues are constant, the condition \eqref{eqvc} clearly implies $\lambda_j<1$ for $1\leq j\leq k$.
The implication of the other direction of the equivalence is trivial. The main novelty of Theorem \ref{thm ke metric} is that $M$ need not be compact and, in addition, the explicit formula \eqref{thm ke metric eq} for the \ke metric.
\end{rmk}

A case of particular interest occurs when the base manifold is a domain in $\bC^n$. Recall that, for a smoothly bounded strongly pseudoconvex domain $\Omega\subset \mathbb{C}^n$ ($n\geq 2$), the Cheng\textendash Yau solution $u\in C^{\infty}(\Omega)$ is the unique solution to \eqref{Dirichlet problem} and $-\log u$ is the potential of a complete \ke metric with negative Ricci curvature.
More generally, when $\Omega \subset \mathbb{C}^n$ is a bounded pseudoconvex domain, it follows from the work of Mok\textendash Yau \cite{MoYau} that $\Omega$ admits a unique complete \ke metric with Ricci curvature equal to $-(n+1)$. If we write $g=\sum_{i, j=1}^n g_{i\bar{j}} dz_i\otimes d\oo{z_j}$ and set $u=(\det g_{i\bar{j}})^{-\frac{1}{n+1}}$, then $u$ satisfies \eqref{Dirichlet problem}. We will call $u$ the Cheng\textendash Yau\textendash Mok solution for $\Omega$. We have the following corollaries of Theorem \ref{thm ke metric}.

\begin{cor}\label{cor CY solution}
	Let $n\geq 1$, $k\geq 1$ and $m=n+k$. Let $D$ be a domain in $\mathbb{C}^n$ and $h$ a positive real analytic function on $D$ such that $\omega:=\sqrt{-1}\,k\,\partial\dbar\log h$ is the K\"ahler form of a complete \k metric $g=\sum_{i, j=1}^n g_{i\bar{j}} dz_i\otimes d\oo{z_j}$ on $D$. Assume that $(D,g)$ has constant Ricci eigenvalues and that all eigenvalues are strictly less than $1$. Consider the domain
	\begin{equation}\label{ball bundle for domain case}
		\Omega:=\{w=(z,\xi) \in D\times \mathbb{C}^k: |\xi|^2 h(z,\bar{z})<1  \},
	\end{equation}
	and the real hypersurface (which is an open dense subset of the boundary of $\Omega$)
	\begin{equation}\label{sphere bundle for domain case}
		\Sigma:=\{w=(z,\xi) \in D\times \mathbb{C}^k: |\xi|^2 h(z,\bar{z})=1  \}.
	\end{equation}
	Then the Cheng\textendash Yau\textendash Mok solution $u$ of $\Omega$ is given by
	\begin{equation}\label{CYM solution}
		u(w)=k^{\frac{n}{m+1}} (GH)^{-\frac{1}{m+1}}\phi(|\xi|h^{\frac{1}{2}}),
	\end{equation}
where $H=h^k$, $G=\det(g_{i\bar{j}})$ and $\phi$ is as in Theorem $\ref{thm ke metric}$. Moreover, $u$ extends real analytically across the boundary piece $\Sigma$.
\end{cor}

\begin{cor}\label{cor homogeneous domain}
	Let $n\geq 1$, $k\geq 1$ and $m=n+k$. Let $D$ be a domain in $\mathbb{C}^n$ and $h$ a positive real analytic function on $D$ such that $\omega:=\sqrt{-1}\,k\,\partial\dbar\log h$ is the K\"ahler form of a complete \k metric $g=\sum_{i, j=1}^n g_{i\bar{j}} dz_i\otimes d\oo{z_j}$ on $D$. Assume that $(D, g)$ is a homogeneous \k manifold (i.e., the group of holomorphic isometries acts transitively on $D$). Consider the domain $\Omega$ and the real hypersurface $\Sigma$ defined by \eqref{ball bundle for domain case} and \eqref{sphere bundle for domain case}.
	Then the Cheng\textendash Yau\textendash Mok solution $u$ of $\Omega$ is given by
	\begin{equation*}
		u(w)=k^{\frac{n}{m+1}} (GH)^{-\frac{1}{m+1}}\phi(|\xi|h^{\frac{1}{2}}),
	\end{equation*}
where $H=h^k$, $G=\det(g_{i\bar{j}})$ and $\phi$ is as in Theorem \ref{thm ke metric}. Moreover, $u$ extends real analytically across the boundary piece $\Sigma$.
\end{cor}

Finally, by studying the potential rationality of the Cheng\textendash Yau\textendash Mok solution, we obtain a characterization of the unit ball in a class of egg domains in terms of the Bergman-Einstein condition. A well-known conjecture posed by Yau \cite{Yau16} asserts that if the Bergman metric
of a bounded pseudoconvex domain is K\"ahler-Einstein, then the domain must be homogeneous. The following result gives an affirmative answer for a class of egg domains.
\begin{prop}\label{prop egg domain}
	Given $p\in \mathbb{R}^+, n\in \mathbb{Z}^+$ and $k\in \mathbb{Z}^+$, we let
	\begin{equation*}
		E_p:=\{(z, \xi)\in \mathbb{C}^n\times \mathbb{C}^k: |z|^2+|\xi|^{2p}<1\}.
	\end{equation*}
	Then the Bergman metric of $E_p$ is \ke if and only if $p=1$.
\end{prop}

\begin{rmk}
	When $n=1$ and $k=1$, the above proposition was proved by Cho \cite{Cho21} (and was known earlier to Fu\textendash Wong \cite{FuWo} if $p$ is an integer). When $k=1$ and $n\geq 1$, it was proved by the authors in \cite{EXX22}.
\end{rmk}


Although we have established versions of Theorems \ref{thm obstruction flat} and \ref{thm ke metric} in the case of line bundles in \cite{EXX22},  we emphasize that the extension to the case of vector bundles carried out in this paper is far from simple and obvious.  To illustrate a basic difference between the two cases, we offer the following observation: If $(L, h)$ is a Hermitian line bundle over a \k manifold $(M, g)$, where $g$ is induced by $-\mathrm{Ric}(L,h)$, then under some mild assumptions (e.g., simple-connectedness of $M$), every holomorphic self-isometry of $(M,g)$ naturally extends to a biholomorphism  of $(L,h)$ which preserves the fiber norms. This statement, however, fails dramatically when $(L,h)$ is replaced by a vector bundle $(E,h)$ of higher rank. While we borrow ideas from \cite{EXX22} to construct the desired metric in the proofs of Theorem \ref{thm obstruction flat} and \ref{thm ke metric}, the main difficulty arises in the verification that the metric constructed indeed satisfies the \ke condition (i.e., the complex Monge-Amp\`ere equation). In particular, in the vector bundle case, the calculation of  the Ricci curvature of the metric we construct seems very difficult in general. A key observation is that the curvature splitting assumption appears to be the right condition to make the computation tractable. Even under this assumption, however, the calculations in the vector bundle case are significantly more involved than those in the line bundle case
(cf.  Lemma \ref{lemma why can we use normal coordinates} and Lemma \ref{lemma compute the Hessian det}).

The paper is organized as follows.  $\S$\ref{Sec preliminaries} gives some preliminary materials on Hermitian vector bundles, including the Griffiths negativity and the curvature split condition. In $\S$\ref{Sec proof of thm obstruction flat} and $\S$\ref{Sec proof of thm ke metric}, we will respectively prove Theorem \ref{thm obstruction flat} and Theorem \ref{thm ke metric}. In $\S$\ref{Sec an interesting case}, Corollary \ref{cor CY solution}, Corollary \ref{cor homogeneous domain} and Proposition \ref{prop egg domain} are established, and some examples of obstruction CR hypersurfaces are constructed.

{\bf Acknowledgment.} The third author would like to thank Professor Yihong Hao and Professor Liyou Zhang for bringing the paper \cite{Coe12} to his attention and for the helpful discussions. 

\section{Hermitian vector bundles}\label{Sec preliminaries}
Let $(E, h)$ be a Hermitian (holomorphic) vector bundle over a complex manifold $M$ of dimension $n$. Denote by $\pi: E\rightarrow M$ the canonical projection and by $E_z=\pi^{-1}(z)$ the fiber at $z$. As before, $\Theta=\Theta_{E, h}$ is the curvature form of the Chern connection of $(E, h)$. Set $k=\rank E$ and let $\{e_{\alpha}(z)\}_{\alpha=1}^k$ be a basis of $E_z$. If $\{e_{\alpha}(z)\}_{\alpha=1}^k$ is orthonormal, then the Ricci curvature of $E$ at $z$ is given by the Hermitian form on $T^{1,0}_zM$:
\begin{equation*}
	\Ric(E) (v, v)=\sum_{\alpha=1}^k \Theta_z(e_{\alpha}(z)\otimes v, e_{\alpha}(z)\otimes v).
\end{equation*}
Note that the condition \eqref{curvature splits} in Definition \ref{def curvature split} implies $	\Ric(E)(v,v)=k H_z(v,v)$.
Hence the curvature $\Theta$ splits if and only if
\begin{equation}\label{H is the Ricci curvature}
	\Theta=\frac{1}{k}h \cdot \Ric(E).
\end{equation}


It will be useful to express the curvature split condition in local coordinates. Let $(D, z)$ with $z=(z_1, \cdots, z_n)$ be a local coordinates chart of $M$ and $\{e_{\alpha}\}_{\alpha=1}^k$ a local frame of $E$ over $D$. Then we have
\begin{equation*}
	\pi^{-1}(D)=\Bigl\{\sum_{\alpha=1}^k \xi_{\alpha} e_{\alpha}(z): z\in D \mbox{ and } \xi_{\alpha}\in \mathbb{C} \mbox{ for any } 1\leq \alpha\leq k \Bigr\},
\end{equation*}
and $(z, \xi)=(z, \xi_1, \cdots, \xi_k)$ form a local coordinates system on $\pi^{-1}(D)$. With respect to the local coordinates $z$ and the local frame $\{e_{\alpha}\}$ over $D$, we can write the curvature tensor $\Theta$ into
\begin{equation*}
	\Theta:=\sum_{\alpha, \beta=1}^k \sum_{i, j=1}^n \Theta_{\alpha\bar{\beta}i\bar{j}}\, e_{\alpha}^* \otimes \oo{e_{\beta}^*}\otimes dz_i \otimes d \oo{z_j}.
\end{equation*}
In addition, if we denote by $	h_{\alpha\bar{\beta}}:=h(e_{\alpha}, e_{\beta})$ and $R_{i\bar{j}}:=\Ric(E) \bigl(\frac{\partial}{\partial z_i}, \frac{\partial}{\partial z_j}\bigr)$,
then it is well-known that
\begin{equation}\label{curvature in local coordinates}
	\Theta_{\alpha\bar{\beta}i\bar{j}}=-\frac{\partial^2 h_{\alpha\bar{\beta}}}{\partial z_i \partial \oo{z_j}}+h^{\gamma\bar{\delta}} \,\frac{\partial h_{\alpha\bar{\delta}}}{\partial z_i}\, \frac{\partial h_{\gamma\bar{\beta}}}{\partial\oo{z_j}},
\end{equation}
and the curvature split condition \eqref{curvature splits} becomes
\begin{equation}\label{curvature splits in local coordinates}
	\Theta_{\alpha\bar{\beta}i\bar{j}}=\frac{1}{k} h_{\alpha\bar{\beta}} R_{i\bar{j}}.
\end{equation}

For the remainder of this paper, we shall use Greek letters $\alpha, \beta, \gamma \cdots$ to denote indices ranging between $1$ and $k$ for the fiber coordinates of $E$,  Roman letters $i, j$ to denote indices ranging between $1$ and $n$ for the local coordinates of $M$, and the Roman letters $s, t$ to denote indices ranging from $1$ to $m=n+k$ for local coordinates of the total space $E$.

The curvature split condition holds true trivially for any Hermitian line bundle. But for Hermitian vector bundles of higher rank this is indeed a strong condition. Nevertheless, we can still construct an abundance of such examples by considering the direct sum of Hermitian line bundles.
\begin{prop}\label{curvature split when the bundle splits prop}
	Let $(L_1, h_1), \cdots, (L_k, h_k)$ be Hermitian line bundles over a complex manifold $M$. Set $(E, h)= (L_1, h_1)\oplus \cdots \oplus (L_k, h_k)$. Then $(E, h)$ is curvature split if and only if $\Ric(L_1, h_1)=\cdots=\Ric(L_k, h_k)$. In particular, if $(L_1, h_1)=\cdots=(L_k, h_k)$, then $(E, h)$ is curvature split.
\end{prop}

\begin{rmk}
	Let $(L_1, h_1)$ and $(L_2, h_2)$ be two Hermitian line bundles over a complex manifold $M$. Then $\Ric(L_1, h_1)=\Ric(L_2, h_2)$ implies that $L_1$ and $L_2$ are smoothly equivalent, but in general, they are not necessarily biholomorphically equivalent.
\end{rmk}

\begin{proof}[Proof of Proposition $\ref{curvature split when the bundle splits prop}$]
	Let $(D, z)$ be a local coordinates chart and $e_{\alpha}$ a local frame of $L_{\alpha}$ for $1\leq k\leq \alpha$ over $D$. Then $\{e_{\alpha} \}_{\alpha=1}^k$ forms a local frame of $E$, in terms of which the metric $h$ becomes the diagonal matrix
	\begin{equation*}
		(h_{\alpha\bar{\beta}})=\begin{pmatrix}
			h_1(e_1, e_1) & &\\
			& \ddots &\\
			& & h_k(e_k, e_k)\\
		\end{pmatrix}.
	\end{equation*}
	By \eqref{curvature in local coordinates}, when $\alpha\neq \beta$, it follows that $\Theta_{\alpha\bar{\beta}i\bar{j}}=0$; when $\alpha=\beta$, we have
	\begin{align*}
		\Theta_{\alpha\bar{\alpha}i\bar{j}}
		=-\frac{\partial^2 h_{\alpha\bar{\alpha}}}{\partial z_i \partial \oo{z_j}}+h^{\alpha\bar{\alpha}} \frac{\partial h_{\alpha\bar{\alpha}}}{\partial z_i} \frac{\partial h_{\alpha\bar{\alpha}}}{\partial \oo{z_j}}
		=-h_{\alpha\bar{\alpha}} \frac{\partial^2\log \bigl(h_{\alpha\bar{\alpha}} \bigr)}{\partial z_i \partial \oo{z_j}} = h_{\alpha\bar{\alpha}}\, \bigl(\Ric(L_{\alpha}, h_{\alpha}) \bigr)_{i\bar{j}}.
	\end{align*}
	As a result,
	\begin{align*}
		\Theta_{\alpha\bar{\beta}i\bar{j}}=h_{\alpha\bar{\beta}} \, \bigl(\Ric(L_{\alpha}, h_{\alpha}) \bigr)_{i\bar{j}}.
	\end{align*}
	On the other hand, the curvature split condition \eqref{curvature splits in local coordinates} writes into
	\begin{equation*}
		\Theta_{\alpha\bar{\beta}i\bar{j}}=\frac{1}{k} h_{\alpha\bar{\beta}}\, \bigl(\Ric(E, h)\bigr)_{i\bar{j}}=\frac{1}{k} h_{\alpha\bar{\beta}} \sum_{\gamma=1}^k \bigl(\Ric({L_{\gamma}}, h_{\gamma}) \bigr)_{i\bar{j}}.
	\end{equation*}
	So the result follows by comparing the above two equations.
\end{proof}


We next recall the notion of and some simple facts about Griffiths negativity (and positivity) for a Hermitian vector bundle.
\begin{defn}
	Let $(E, h)$ be a Hermitian vector bundle over a complex manifold. We say $(E, h)$ is Griffiths negative (resp. positive) if for any $z\in M, e\in E_z$ and $v\in T^{1,0}M$ we have
	\begin{equation*}
		\Theta_z(e\otimes v, e\otimes v) \leq 0 \quad (\mbox{resp.} \geq 0),
	\end{equation*}
and the equality holds if and only if $e=0$ or $v=0$ (i.e., $e\otimes v=0$). In local coordinates, this means that for any $v=(v_1, \cdots, v_{n}) \in \mathbb{C}^n$ and $\xi=(\xi_1, \cdots, \xi_k) \in \mathbb{C}^k$ we have
$\Theta_{\alpha\bar{\beta}i\bar{j}} \xi_{\alpha} \oo{\xi_{\beta}} v_{i} \oo{v_j} \leq 0~(\mbox{resp.} \geq 0),$
and the equality holds if and only if $v=0$ or $\xi=0$.
\end{defn}

\begin{lemma}\label{Griffiths negative and determinant bundle lemma}
Let $(E, h)$ be a Hermitian vector bundle over a complex manifold. 	Consider the properties:
	\begin{itemize}
		\item[(1)] The Hermitian vector bundle $(E, h)$ is Griffiths negative (resp. positive). 		
		\item[(2)] The determinant line bundle $L=\det E$ with the determinant metric $\det h$ is negative (resp. positive).
	\end{itemize}
In general, it holds that $(1)$ implies $(2)$. If the curvature $\Theta$ splits, then $(1)$ is equivalent to $(2)$.
\end{lemma}

\begin{proof}
	The fact that (1) implies (2) follows directly from the identity
	\begin{equation*}
		\bigl(\Ric(L)\bigr)_{i\bar{j
		}}=\bigl(\Ric(E)\bigr)_{i\bar{j}}=h^{\alpha\bar{\beta}}\Theta_{\alpha\bar{\beta}i\bar{j}}.
	\end{equation*}
When $\Theta$ splits, by \eqref{curvature splits in local coordinates} we have
$\Theta_{\alpha\bar{\beta}i\bar{j}}=\frac{1}{k} h_{\alpha\bar{\beta}} \, \bigl(\Ric(L)\bigr)_{i\bar{j}}.$
Then it is clear that (2) also implies (1) in this case.
\end{proof}

\begin{rmk}\label{Griffiths negative if and only if each line bundle is negative rmk}
	Let $(L_1, h_1), \cdots, (L_k, h_k)$ be line bundles over a complex manifold $M$, and consider the vector bundle $(E, h):= (L_1, h_1)\oplus \cdots \oplus (L_k, h_k)$. Then it follows that $\Ric(\det E, \det h)=\Ric(E, h)=\sum_{j=1}^k\Ric(L_j, h_j)$. By using Proposition \ref{curvature split when the bundle splits prop} and Lemma \ref{Griffiths negative and determinant bundle lemma} we immediately obtain that if $(E, h)$ is curvature split, then $(E, h)$ is Griffiths negative if and only if each $(L_j, h_j)$ for $1\leq j\leq k$ is negative.
\end{rmk}

\begin{rmk}
	Consider the special case when $M$ is a Riemann surface (i.e., $n=1$). Suppose $(E, h)$ is Griffiths negative vector bundle over $M$ and let $g$ be the metric induced by $-\Ric(E, h)$. Then by \eqref{curvature splits in local coordinates}, the vector bundle $(E, h)$ is curvature split if and only if $(E, h)$ is Hermitian-Einstein.
\end{rmk}


For a Hermitian holomorphic vector bundle $(E,h)$, recall that $S(E):=\{e\in E: |e|_h=1 \}$ denotes the sphere bundle of $(E, h)$. We conclude this section by noting the following fundamental fact;  see \cite[Lemma 1.2.2]{MoNg} and \cite[Proposition 5.3]{Coe12}.
\begin{prop}
	If $(E, h)$ is Griffiths negative, then $S(E)$ is strongly pseudoconvex.
\end{prop}

\section{Proof of Theorem \ref{thm obstruction flat}} \label{Sec proof of thm obstruction flat}
In this section we prove Theorem \ref{thm obstruction flat}. As the obstruction flatness is a local property (cf. \cite{Graham1987a}), we need to show that for any point $p\in S(E)$, there exists some neighborhood $U$ of $p$ on $E$ such that $S(E)\cap U$ is obstruction flat. By the work of Graham \cite{Graham1987a} again, it suffices to construct a function $u\in C^{\infty}(U)$ such that $u=0$ on $S(E)\cap U$ and $J(u)=1$ on the pseudoconvex side of  $S(E)\cap U$. To do this,  we shall establish a series of propositions and lemmas.

\begin{prop}\label{prpnzh1}
	Let $P(y)$ be a monic polynomial in $y\in \mathbb{R}$ of degree $m-1\geq 1$ and $Q(y)$ a polynomial satisfying $\frac{dQ}{dy}=(m+1)yP(y)$ (thus $Q$ is a monic polynomial of degree $m+1$ and is unique up to a constant). Suppose $\hat{P}$ and $\hat{Q}$ are polynomials defined by
	\begin{equation*}
		\hat{P}(x)=x^{m-1}P(x^{-1}), \quad \hat{Q}(x)=x^{m+1} Q(x^{-1}).
	\end{equation*}
	Let $I \subset (0, \infty)$ be an open interval containing $r=1$ and  $Z(r)$ a real analytic function on $I$ satisfying the following conditions:
	\begin{equation}\label{eqnzodenz1}
		rZ'\hat{P}(Z)+\hat{Q}(Z)=0~\text{on}~I, \quad Z(1)=0.
	\end{equation}
	\begin{equation}\label{eqnzivnz}
		Z'(r)<0~\text{and}~\hat{P}(Z)>0~\text{on}~I, ~Z(r)>0 ~\text{on}~I_0:=I \cap (0,1), ~\text{and}~ Z'(1)=-1.
	\end{equation}
	Let $k\in \mathbb{Z}^+$ with $k\leq m-1$ and set $\phi(r)=2\bigl(\frac{r^{2k-1}}{-Z' \hat{P}(Z)}\bigr)^{\frac{1}{m+1}}Z$ on $I$. Then $\phi$ is real analytic on $I$, $\phi(1)=0$ and $\phi>0$ on $I_0$. Moreover, $\phi$ satisfies that $(m+1)rZ\phi'+(m+1-2kZ)\phi=0$ on $I$, and $\phi'(1)=-2$.
\end{prop}

\begin{rmk}
	Since $P$ and $Q$ are monic polynomials, the polynomials $\hat{P}$ and $\hat{Q}$ satisfy $\hat{P}(0)=1$ and $\hat{Q}(0)=1.$ By elementary ODE theory, (\ref{eqnzodenz1}) has a real analytic solution $Z$ which satisfies \eqref{eqnzivnz} in some open interval $I$ containing $1$.
\end{rmk}

\begin{proof}
	It is clear that $\phi$ is real analytic on $I, ~\phi(1)=0,$ and $\phi>0$ on $I_0$ by the definition of $\phi$ and the assumption of $Z$. We only need to prove the last assertion in Proposition \ref{prpnzh1}. The proof is similar to that of Proposition 2.1 in \cite{EXX22}. We just highlight the following two lemmas and the conclusion will be evident.
	
\begin{lemma}\label{lemma3d3}
	Let $Z$ be as in Proposition \ref{prpnzh1}. Then we have
	$$r(Z'Z^{-(m+1)}\hat{P}(Z))'=(m+1)Z'Z^{-(m+2)}\hat{P}(Z)-Z'Z^{-(m+1)}\hat{P}(Z)~\text{on}~I_0.$$
	Equivalently,
	$$r\frac{(Z'Z^{-(m+1)} \hat{P}(Z))'}{Z'Z^{-(m+1)}\hat{P}(Z)}=-1+(m+1)Z^{-1}~\text{on}~I_0.$$
\end{lemma}
		

\begin{lemma}\label{lemma3d4}
	Let $\phi$ and $Z$ be as in Proposition \ref{prpnzh1}. Then we have
	$$r\frac{(Z'Z^{-(m+1)}\hat{P}(Z))'}{Z'Z^{-(m+1)}\hat{P}(Z)}=(2k-1)-(m+1)r\frac{\phi'}{\phi}~\text{on}~I_0.$$
\end{lemma}


Lemma \ref{lemma3d3} is identical to Lemma 2.3 in \cite{EXX22}; Lemma \ref{lemma3d4} follows from a similar argument as that of Lemma 2.4 in \cite{EXX22}. Thus, we omit their proofs. Finally we compare (the second equation in) Lemma \ref{lemma3d3} and Lemma \ref{lemma3d4} to obtain $(m+1)rZ\phi'+(m+1-2kZ)\phi=0$ on $I_0$. By analyticity, it actually holds on $I.$ Recall
$Z(1)=0, \hat{P}(0)=1$ and $Z'(1)=-1$. It then follows from the definition of $\phi$ that $\phi'(1)=-2.$ This finishes the proof of Proposition \ref{prpnzh1}.
\end{proof}

Let $(M, g)$ and $(E, h)$ be as in Theorem \ref{thm obstruction flat}. Denote by $n$ the complex dimension of $M$ and by $k$ the rank of $E$. Choose a coordinates chart $(D, z)$ of $M$ with a local frame $\{e_{\alpha}\}_{\alpha=1}^k$ of $E$ over $D$. Let $\pi: E\rightarrow M$ be the canonical projection. Then we have
\begin{equation*}
	\pi^{-1}(D)=\Bigl\{\sum_{\alpha=1}^k \xi_{\alpha}e_{\alpha}(z): (z, \xi)\in D\times\mathbb{C}^k \Bigr\}.
\end{equation*}
Under this trivialization, the sphere bundle $S(E)$ over $D$ can be written as
\begin{equation*}
	\Sigma:=S(E)\cap \pi^{-1}(D)=\Bigl\{(z, \xi)\in D\times \mathbb{C}^k: \sum_{\alpha, \beta=1}^k h_{\alpha\bar{\beta}}(z,\bar{z})\xi_{\alpha}\oo{\xi_{\beta}}=1  \Bigr\},
\end{equation*}
where $h_{\alpha\bar{\beta}}(z,\bar{z})=h(e_{\alpha}(z), e_{\beta}(z))$ for $1\leq \alpha, \beta\leq k$. In the local coordinates $z=(z_1, \cdots, z_n)$ we write $g=\sum_{i, j=1}^n g_{i\bar{j}} dz_i\otimes d\oo{z_j}$ on $D$. As $g$ is induced by $-\Ric(E, h)$, we have $g_{i\bar{j}}=\frac{\partial^2 \log H}{\partial z_i \partial \oo{z_j}}$ where $H=\det(h_{\alpha\bar{\beta}})$. Let $G=\det (g_{i\Ol{j}})$ on $D$.

Set $m=n+k$ and denote by $I_n$ the $n \times n$ identity matrix. Given $p\in M$, let $T(y, p)$ be the characteristic polynomial of the linear operator $\frac{2k}{m+1} \Ric \cdot g^{-1}: T_{p}^{1,0}M \rightarrow T_{p}^{1,0}M.$ That is,
\begin{equation}\label{T characteristic poly}
	T(y, p)=\det(yI_n-\frac{2k}{m+1} \mathrm{Ric} \cdot g^{-1}).
\end{equation}
In the local coordinates $z=(z_1, \cdots, z_n)$, writing the Ricci tensor as $\mathrm{Ric}=(R_{i \Ol{k}})_{1 \leq i, k \leq n}=-(\frac{\partial^2 \log G}{\partial z_i \partial \Ol{z_k}})_{1 \leq i, k \leq n},$ we see that $T(y, p)$ is the determinant of the $n \times n$ matrix $(y\delta_{ij}-\frac{2k}{m+1}R_{i \Ol{k}} \cdot g^{j \Ol{k}}(p))$.

Now we define
\begin{equation}\label{P polynomial def}
	P(y, p):= (y-\frac{2k}{m+1})^{k-1}T(y, p).
\end{equation}
By the constant Ricci eigenvalue assumption, 
$P(y,p)$ does not depend on $p$. We will therefore just denote it by $P(y).$
It is clear that $P(y)$ is a monic polynomial in $y$ of degree $m-1$. We apply Proposition \ref{prpnzh1} to this polynomial $P(y)$ (with $k$ equal to the rank of $E$) to obtain  polynomials $Q(y), \hat{P}(x), \hat{Q}(x)$, as well as real analytic functions $Z(r)$ and $\phi(r)$ in some interval $I$ containing $r=1$. We let $y(r)=\frac{1}{Z(r)}$ for $r \in I_0.$ By Proposition \ref{prpnzh1}, $(m+1)rZ\phi'+(m+1-2kZ)\phi=0$ on $I$. It then follows that
\begin{equation}\label{eqn3d5}
	y(r)=\frac{2k\phi(r)-(m+1)r\phi'(r)}{(m+1)\phi(r)}=\frac{2k}{m+1}-r\frac{\phi'}{\phi} \quad \mbox{ on } I_0.
\end{equation}

Theorem \ref{thm obstruction flat} will follow from the next proposition.
\begin{prop}\label{prpn3d5}
	Let
	\begin{align*}
		U=\Bigl\{w:=(z, \xi) \in D\times \mathbb{C}^{k}: |w|_h \in I \Bigr\}, \qquad
		U_0=\Bigl\{w:=(z, \xi) \in D\times \mathbb{C}^{k}: |w|_h \in I_0 \Bigr\},
	\end{align*}
	where $|w|_h^2=\sum_{\alpha, \beta=1}^k h_{\alpha\bar{\beta}}(z, \bar{z}) \xi_{\alpha}\oo{\xi_{\beta}}$.
	Set
	\begin{equation}\label{CY solution}
		u(w)=k^{\frac{n}{m+1}}(GH)^{-\frac{1}{m+1}}\phi\big(|w|_h\big)~\text{for}~w  \in U.
	\end{equation}
Here $G(w)$ is understood as $G(z)$ for $w=(z, \xi).$ Likewise for $H$. Then $u$ is smooth in $U$ and satisfies
	$$J(u)=1~\text{on}~ U_0, \quad u=0~\text{on}~\Sigma.$$
	Consequently, $\Sigma$ is obstruction flat.
\end{prop}

\begin{proof}
	The smoothness of $u$ follows easily from the that of $\phi$, as well as that of $G, H$ and $h$. 
	We thus only need to prove the remaining assertions. For that,  we first prove the following lemma. Set $X=X(w)=|w|_h$ for $w \in U$,  and
	\begin{align}\label{Y}
		Y=Y(w)=\frac{2k}{m+1}-X\frac{\phi'(X)}{\phi(X)} \quad \mbox{ for } w \in U_0.
	\end{align}
	Then by (\ref{eqn3d5}), we have $Y=y(r)|_{r=X}=\frac{1}{Z(r)}|_{r=X}$. Note that $X$ and $Y$ are independent of the choice of local coordinates and local frame.
	
	In the following, we will also write the coordinates $w$ of $D \times \mathbb{C}^k$ as $(w_1, \cdots, w_{m-1}, w_m).$ That is, we identify $w_i$ with $z_i$ for $1 \leq i \le n$, and $w_{n+\alpha}$ with $\xi_{\alpha}$ for $1\leq \alpha\leq k$. For a sufficiently differentiable function $\Phi$ on an open subset of $D \times \mathbb{C}^k$,  we write, for $1 \leq s, t \leq m,$ $\Phi_s=\frac{\partial \Phi}{\partial w_s}, \Phi_{\Ol{t}}=\frac{\partial \Phi}{\partial \Ol{w_t}},$ and $\Phi_{s\Ol{t}}=\frac{\partial^2 \Phi}{\partial w_s \partial \Ol{w_t}}$. To simplify the later computations, we introduce the following lemma.
	
\begin{lemma}\label{lemma why can we use normal coordinates}
	Let $\pi: (E, h)\rightarrow (M, g)$ be a Hermitian vector bundle of rank $k$ over an $n$-dimensional Hermitian manifold $(M, g)$. Write $m=n+k$. Let $\Omega$ be a smooth $(m ,m)$-form on $E$ (regarded as an $m$-dimensional complex manifold) (resp. an open subset $V\subset E$). Then we can define a smooth function $\Phi$ on $E$ (resp. on $V$) in the following manner: 
	Pick any local coordinates chart $(D, z)$ of $M$ and any local frame $\{e_{\alpha}\}_{\alpha=1}^k$ of $E$ over $D$. This induces a natural system of coordinates $w=(z, \xi)$ for $E$ on $\pi^{-1}(D)$ with
	\begin{equation*}
		\pi^{-1}(D)=\Bigl\{\sum_{\alpha=1}^k\xi_{\alpha} e_{\alpha}(z): (z, \xi)\in D\times \mathbb{C}^k \Bigr\}.
	\end{equation*}	
	In the above coordinates, we write $g=\sum_{i, j=1}^n g_{i\bar{j}} dz_i\otimes d\oo{z_j}$, $h_{\alpha\bar{\beta}}=h(e_{\alpha}, e_{\beta})$ and $\Omega=\sigma(z, \xi)dz\wedge d\xi\wedge d\oo{z}\wedge d\oo{\xi}$ where $dz=dz_1\wedge\cdots dz_n$ and $d\xi=d\xi_1\wedge\cdots\wedge d\xi_k$. We define the function $\Phi: \pi^{-1}(D)\rightarrow \mathbb{C}$ by
	\begin{equation*}
		\Phi(z, \xi)=\frac{\sigma(z,\xi)}{\det \bigl(g_{i\bar{j}}(z)\bigr) \cdot \det \bigl(h_{\alpha\bar{\beta}}(z)\bigr)}.
	\end{equation*}
	Then $\Phi$ is independent of the choice of local coordinates and local frame. As a result, $\Phi$ is a well-defined function on $E$ (resp. on $V$).
\end{lemma}

\begin{proof}
	To show $\Phi$ is well-defined, we take two local coordinates charts, $(D, z)$ and $(\widetilde{D}, \widetilde{z})$ of $M$, and two local frame of $E$, $\{e_{\alpha}\}_{\alpha=1}^k$ over $D$ and $\{\widetilde{e}_{\alpha}\}_{\alpha=1}^k$ over $\widetilde{D}$. Then we obtain two induced coordinates system, $w=(z, \xi)$ on $\pi^{-1}(D)$ and $\widetilde{w}=(\widetilde{z}, \widetilde{\xi})$ on $\pi^{-1}(\widetilde{D})$. On $D\cap \widetilde{D}$ we can write $\widetilde{z}=\varphi(z)$ for some biholomorphism $\varphi$ and $\widetilde{\xi}=\xi\cdot A(z)$ for some holomorphic map $A: D\cap \widetilde{D} \rightarrow \Gl(k, \mathbb{C})$. By the relation $g=\sum_{i, j=1} g_{i\bar{j}}(z) dz_i\otimes d\oo{z_j}=\sum_{i, j=1}^n \widetilde{g}_{i\bar{j}}(\widetilde{z}) d\widetilde{z}_i\otimes d\oo{\widetilde{z}_j}$, we have
	\begin{equation*}
		\det\bigl(g_{i\bar{j}}(z) \bigr)=\det\bigl(\widetilde{g}_{i\bar{j}}(\widetilde{z})\bigr) \cdot \bigl| J\varphi(z)\bigr|^2,
	\end{equation*}
	where $J(\varphi)$ is the determinant of the Jacobian matrix of $\varphi$. Similarly, we also have
	\begin{equation*}
		\det\bigl(h_{\alpha\bar{\beta}}(z) \bigr)=\det\bigl(\widetilde{h}_{\alpha\bar{\beta}}(\widetilde{z})\bigr)\cdot \bigl|\det A(z)\bigr|^2.
	\end{equation*}
	On the other hand, by writing $\Omega=\sigma(z, \xi)dz\wedge d\xi\wedge d\oo{z}\wedge d\oo{\xi}=\widetilde{\sigma}(\widetilde{z}, \widetilde{\xi})d\widetilde{z}\wedge d\widetilde{\xi}\wedge d\oo{\widetilde{z}}\wedge d\oo{\widetilde{\xi}}$ on $\pi^{-1}(D)\cap \pi^{-1}(\widetilde{D})$, we get
	\begin{equation*}
		\sigma(z,\xi)=\widetilde{\sigma}(\widetilde{z}, \widetilde{\xi})\cdot \bigl|J\varphi(z)\bigr|^2\cdot \bigl|\det A(z) \bigr|^2.
	\end{equation*}
	It follows that
	\begin{equation*}
		\frac{\sigma(z, \xi)}{\det\bigl(g_{i\bar{j}}(z) \bigr)\cdot \det\bigl(h_{\alpha\bar{\beta}}(z) \bigr)}=\frac{\widetilde{\sigma}(\widetilde{z}, \widetilde{\xi})}{\det\bigl(\widetilde{g}_{i\bar{j}}(\widetilde{z})\bigr)\cdot \det\bigl(\widetilde{h}_{\alpha\bar{\beta}}(\widetilde{z})\bigr)} \quad \mbox{on } \pi^{-1}(D)\cap \pi^{-1}(\widetilde{D}).
	\end{equation*}
	So the proof is completed.
\end{proof}

Next we will prove the following lemma on the computation of $\det\bigl((-\log u)_{s\bar{t}}\bigr)$.
\begin{lemma}\label{lemma compute the Hessian det}
	Let u(w) be as defined in Proposition \ref{prpn3d5}.  Write $Y'=\frac{dY}{dX},$ i.e., $Y'=y'(r)|_{r=X}$. Then the following is a well-defined function on $U_0$:
	\begin{equation*}
		\Phi(w)=\frac{\det\bigl((-\log u)_{s\bar{t}}(w)\bigr)}{\det\bigl(g_{i\bar{j}}(z)\bigr)\cdot\det\bigl(h_{\alpha\bar{\beta}}(z)\bigr)}.
	\end{equation*}
	That is, $\Phi$ is independent of the choice of local coordinates of $M$ and local frame of $E$. Moreover,
	\begin{equation*}
		\Phi=\frac{P(Y)Y'}{2^{m+1}k^nX^{2k-1}} \quad \mbox{on } U_0.
	\end{equation*}
\end{lemma}

\begin{proof}
	Consider the Hermitian $(1, 1)$-form $\widetilde{\omega}=-\sqrt{-1}\partial\dbar\log u$ on $U_0\subset E$. 
	In the above local coordinates $w=(z, \xi)$, we can write $\widetilde{\omega}=\sqrt{-1}\sum_{s, t=1}^m (-\log u)_{s\bar{t}} dw_s\wedge d\oo{w_t}$. Then $\frac{\widetilde{\omega}^m}{m!}$ is a smooth $(m, m)$-form on $U_0$ and
	\begin{equation*}
		\frac{\widetilde{\omega}^m}{m!}=(\sqrt{-1})^m \det\bigl((-\log u)_{s\bar{t}}\bigr) dw_1\wedge d\oo{w_1}\wedge \cdots \wedge dw_m\wedge d\oo{w_m}.
	\end{equation*}
	By Lemma \ref{lemma why can we use normal coordinates}, $\Phi$ is well-defined on $U_0$.
	
	To prove the second assertion, we fix a point $p\in E$ and write $q=\pi(p)$. To compute the value of $\Phi$ at $p$, by the first part of the lemma, we can use any local coordinates of $M$ and any local frame of $E$. In particular, we can choose a local coordinates of $M$ and a local frame of $E$ at $q$ such that the induced coordinates of $E$ near $p$, which we still denote by $w=(z, \xi)$, satisfy $z(q)=0, h_{\alpha\bar{\beta}}(q)=\delta_{\alpha\beta}$ and $d h_{\alpha\bar{\beta}}(q)=0$, where $\delta_{\alpha\beta}$ is the Kronecker delta. Under the above coordinates, the curvature \eqref{curvature in local coordinates} at point $q$ simplifies into
	\begin{equation*}
		\Theta_{\alpha\bar{\beta}i\bar{j}}(0)=-\frac{\partial^2 h_{\alpha\bar{\beta}}}{\partial z_i \partial \oo{z_j}}(0).
	\end{equation*}
	We take the logarithm of $u$ (which is defined in \eqref{CY solution}) and obtain
	\begin{equation*}
		-\log u=-\frac{n}{m+1}\log k+\frac{1}{m+1} \log G+\frac{1}{m+1} \log H-\log \phi(X),
	\end{equation*}
	where $X=|w|_h=\bigl(\sum_{\alpha, \beta=1}^k h_{\alpha\bar{\beta}}(z, \bar{z}) \xi_{\alpha}\oo{\xi_{\beta}}\bigr)^{1/2}$. A straightforward computation yields
	\begin{equation*}
		(-\log u)_{s\bar{t}}=\frac{1}{m+1}\bigl(\log G\bigr)_{s\bar{t}}+\frac{1}{m+1}\bigl(\log H\bigr)_{s\bar{t}}-\Bigl(\frac{\phi'}{\phi}X_{s\bar{t}}+\bigl(\frac{\phi'}{\phi}\bigr)'X_sX_{\bar{t}}\Bigr) \quad \mbox{for any $1\leq s, t \leq m$}.
	\end{equation*}
	Since $X_{s\bar{t}}=X(\log X)_{s\bar{t}}+\frac{1}{X} X_s X_{\bar{t}}$, the above writes into
	\begin{equation*}
		(-\log u)_{s\bar{t}}=\frac{1}{m+1}\bigl(\log G\bigr)_{s\bar{t}}+\frac{1}{m+1}\bigl(\log H\bigr)_{s\bar{t}}-X\bigl(\frac{\phi'}{\phi}\bigr) \bigl(\log X\bigr)_{s\bar{t}}-\Bigl(\frac{1}{X}\bigl(\frac{\phi'}{\phi}\bigr)+\bigl(\frac{\phi'}{\phi}\bigr)'\Bigr)X_sX_{\bar{t}}.
	\end{equation*}
	To continue the computation of this Hessian matrix, we shall divide it into the following three cases.
	
\textbf{Case 1.} $1\leq s, t\leq n$.

For this case, we have $w_s=z_s$ and $w_t=z_t$. We denote $i=s$ and $j=t$ for simplicity. At $w=(0, \xi)$, by the facts $h_{\alpha\bar{\beta}}=\delta_{\alpha\bar{\beta}}$ and $d h_{\alpha\bar{\beta}}=0$ it follows that $X=|\xi|$ and $X_i=X_{\bar{j}}=0$. Moreover, we also have
\begin{align*}
	\bigl(\log X\bigr)_{i\bar{j}}\big|_{w=(0, \xi)}=\frac{1}{2}\bigl( \log X^2\bigr)_{i\bar{j}}\big|_{w=(0, \xi)}=\frac{1}{2}\frac{(X^2)_{i\bar{j}}}{X^2}\Big|_{w=(0, \xi)}
	=&\frac{1}{2|\xi|^2} \sum_{\alpha, \beta=1}^k \frac{\partial h_{\alpha\bar{\beta}}}{\partial z_i \partial \oo{z_j}}(0) \xi_{\alpha} \oo{\xi_{\beta}}\\
	=&\frac{1}{2|\xi|^2}\sum_{\alpha, \beta=1}^k \bigl(-\Theta_{\alpha\bar{\beta}i\bar{j}}(0)\bigr) \xi_{\alpha}\oo{\xi_{\beta}}.
\end{align*}
Since the curvature $\Theta$ splits and the \k metric is induced by $-\Ric(E, h)$ by the assumption of Theorem \ref{thm obstruction flat}, it follows that
\begin{equation*}
	\Theta_{\alpha\bar{\beta}i\bar{j}}(0)=\frac{1}{k} h_{\alpha\bar{\beta}}(0) \cdot \bigl(\Ric(E, h)\bigr)_{i\bar{j}}(0)=-\frac{1}{k} \delta_{\alpha\beta}\cdot  g_{i\bar{j}}(0).
\end{equation*}
Therefore,
\begin{equation*}
	(\log X)_{i\bar{j}}\big|_{w=(0, \xi)}=\frac{1}{2k} g_{i\bar{j}}(0).
\end{equation*}
Denote the Ricci form $\Ric(g)=\sqrt{-1}R_{i\bar{j}} dz_i\wedge d\oo{z_j}$. Then $(\log G)_{i\bar{j}}=-R_{i\bar{j}}$, $(\log H)_{i\bar{j}}=g_{i\bar{j}}$ and by \eqref{Y} we have
\begin{equation*}
	(-\log u)_{s\bar{t}}\big|_{w=(0, \xi)}=-\frac{1}{m+1}R_{i\bar{j}}(0)+\frac{1}{m+1}g_{i\bar{j}}(0)-\frac{X}{2k} \frac{\phi'}{\phi}  g_{i\bar{j}}(0)=-\frac{1}{m+1}R_{i\bar{j}}(0)+\frac{Y}{2k}  g_{i\bar{j}}(0).
\end{equation*}

\textbf{Case 2.} $1\leq s\leq n, n+1\leq t\leq n+k$ or $n+1\leq s\leq n+k, 1\leq t\leq n$.

For this case, a similar computation yields $(\log G)_{s\bar{t}}$ and $(\log H)_{s\bar{t}}$ vanish identically and $(\log X)_{s\bar{t}}|_{w=(0, \xi)}=X_sX_{\bar{t}}|_{w=(0, \xi)}=0$. Therefore, $(-\log u)_{s\bar{t}}|_{w=(0, \xi)}=0$.

\textbf{Case 3.} $n+1\leq s, t\leq n+k$.

For this case, we have $w_s=\xi_{s-n}$ and $w_t=\xi_{t-n}$. We denote $\alpha=s-n$ and $\beta=t-n$ for simplicity. Then a straightforward computation yields that at any $w=(0, \xi)\in U_0$ we have $X=|\xi|$, $X_s=\frac{1}{2}|\xi|^{-1}\oo{\xi_{\alpha}}$, $\oo{X_t}=\frac{1}{2}|\xi|^{-1}\xi_{\beta}$ and
\begin{equation*}
	\bigl(\log X\bigr)_{s\bar{t}}=\frac{1}{2X^2} \delta_{\alpha\beta}-\frac{1}{2X^4}\oo{\xi_{\alpha}}\xi_{\beta}.
\end{equation*}
As $(\log G)_{s\bar{t}}$ and $(\log H)_{s\bar{t}}$ vanish identically in this case, we have at $w=(0, \xi)$
\begin{equation*}
	(-\log u)_{s\bar{t}}=-X\bigl(\frac{\phi'}{\phi}\bigr) \Bigl(\frac{1}{2X^2} \delta_{\alpha\beta}-\frac{1}{2X^4}\oo{\xi_{\alpha}}\xi_{\beta}\Bigr)-\Bigl(\frac{1}{X}\bigl(\frac{\phi'}{\phi}\bigr)+\bigl(\frac{\phi'}{\phi}\bigr)'\Bigr)\frac{\oo{\xi_{\alpha}}\xi_{\beta}}{4X^2}.
\end{equation*}
By \eqref{Y}, we further write it into
\begin{align*}
	(-\log u)_{s\bar{t}}\,\big|_{w=(0, \xi)}=&\bigl(Y-\frac{2k}{m+1}\bigr) \Bigl(\frac{1}{2X^2} \delta_{\alpha\beta}-\frac{1}{2X^4}\oo{\xi_{\alpha}}\xi_{\beta}\Bigr)+\frac{Y'}{X}\frac{\oo{\xi_{\alpha}}\xi_{\beta}}{4X^2}\\
	=&\frac{1}{2X^2}\bigl(Y-\frac{2k}{m+1}\bigr)\delta_{\alpha\beta}+\bigl(\frac{XY'}{4}-\frac{Y}{2}+\frac{k}{m+1}\bigr)\frac{\oo{\xi_{\alpha}}\xi_{\beta}}{X^4}.
\end{align*}

Combining the above three cases, we see the complex Hessian matrix $\bigl((-\log u)_{s\bar{t}}\bigr)$ at $w=(0, \xi)$ is block diagonal. Moreover,
\begin{equation}\label{determinant Hessian of log u}
	\det\bigl((-\log u)_{s\bar{t}}\bigr)_{1\leq s, t\leq m}=\det\bigl((-\log u)_{s\bar{t}}\bigr)_{1\leq s, t\leq n} \cdot \det\bigl((-\log u)_{s\bar{t}}\bigr)_{n+1\leq s, t\leq m}.
\end{equation}
Now we need to compute the determinants appearing on the right hand side of the above equation. By the above computation in Case 1, we have at $w=(0, \xi)$
\begin{align*}
	\det\bigl((-\log u)_{i\bar{j}}\bigr)_{1\leq i, j\leq n}=&\det\Bigl(-\frac{1}{m+1}R_{i\bar{j}}(0)+\frac{Y}{2k}  g_{i\bar{j}}(0)\Bigr)\\
	=&\bigl(\frac{1}{2k}\bigr)^n\det\Bigl(Y\delta_{ij}-\frac{2k}{m+1}R_{i\bar{k}}(0) \,g^{j\bar{k}}(0)\Bigr) \cdot \det\bigl(g_{i\bar{j}}(0)\bigr)\\
	=& \bigl(\frac{1}{2k}\bigr)^n\, T(Y)\, G,
\end{align*}
where the last equality follows from \eqref{T characteristic poly}. For the second determinant on the right hand side of \eqref{determinant Hessian of log u}, by the computation in Case 3, we have at $w=(0, \xi)$,
\begin{equation*}
	\det\bigl((-\log u)_{s\bar{t}}\bigr)_{n+1\leq s, t\leq m}=\det\Bigl(\frac{1}{2X^2}\bigl(Y-\frac{2k}{m+1}\bigr)\delta_{\alpha\beta}+\bigl(\frac{XY'}{4}-\frac{Y}{2}+\frac{k}{m+1}\bigr)\frac{\oo{\xi_{\alpha}}\xi_{\beta}}{X^4}\Bigr)_{1 \leq  \alpha, \beta \leq k}.
\end{equation*}
As $X=|\xi|$ at $w=(0, \xi)$, by the matrix determinant lemma we get
\begin{align*}
	\det\bigl((-\log u)_{s\bar{t}}&\bigr)_{n+1\leq s, t\leq m}\\
	=&\Bigl(\frac{1}{2X^2}\bigl(Y-\frac{2k}{m+1}\bigr)\Bigr)^{k-1}\Bigl(\frac{1}{2X^2}\bigl(Y-\frac{2k}{m+1}\bigr)+\bigl(\frac{XY'}{4}-\frac{Y}{2}+\frac{k}{m+1}\bigr)\frac{1}{X^2}\Bigr)\\
	=&\frac{Y'}{2^{k+1}X^{2k-1}}\bigl(Y-\frac{2k}{m+1}\bigr)^{k-1}.
\end{align*}
We now plug these results back into \eqref{determinant Hessian of log u} and further use \eqref{P polynomial def} to obtain that at $w=(0, \xi)$
\begin{align*}
	\det\bigl((-\log u)_{s\bar{t}}\bigr)_{1\leq s, t\leq m}=\frac{Y'}{2^{m+1}k^nX^{2k-1}} P(Y)\, G.
\end{align*}
Note that at $w=(0, \xi)$ we have $H=1$. Therefore,
\begin{equation*}
	\frac{\det\bigl((-\log u)_{s\bar{t}}\bigr)_{1\leq s, t\leq m}}{GH}=\frac{P(Y) Y'}{2^{m+1} k^n X^{2k-1}}.
\end{equation*}
This proves Lemma \ref{lemma compute the Hessian det}.
\end{proof}

Now we resume the proof of Proposition \ref{prpn3d5}. By the definition of $\phi$ in Proposition \ref{prpnzh1}, for any $r\in I_0$ we have
\begin{equation*}
	-Z'\hat{P}(Z) Z^{-(m+1)}=2^{m+1} r^{2k-1} \phi^{-(m+1)}(r)
\end{equation*}
Recall that $y(r)=\frac{1}{Z(r)}$. We thus get
\begin{equation*}
	y'\hat{P}(y^{-1}) y^{m-1}=2^{m+1} r^{2k-1} \phi^{-(m+1)}(r).
\end{equation*}
By the definition of $\hat{P}$ in Proposition \ref{prpnzh1}, we have $P(y)=y^{m-1} \hat{P}(y^{-1})$. Therefore,
\begin{equation*}
	y'P(y)=2^{m+1} r^{2k-1} \phi^{-(m+1)}(r).
\end{equation*}
Now we take $r=X=|w|_h$ for $w=(z, \xi)\in U_0$. Then
\begin{equation*}
	Y' P(Y)=2^{m+1} X^{2k-1} \phi^{-(m+1)}(X)=2^{m+1}k^n X^{2k-1} (GH)^{-1} u^{-(m+1)}(X),
\end{equation*}
where the second equality follows from the definition of $u$. Therefore,
\begin{equation*}
	\frac{Y'P(Y)}{2^{m+1}k^n X^{2k-1}}=(GH)^{-1} u^{-(m+1)}(X).
\end{equation*}
By Lemma \ref{lemma compute the Hessian det} we obtain
$\det \bigl((-\log u)_{s\bar{t}}\bigr)_{1 \leq s, t \leq m}=u^{-(m+1)} \quad \mbox{ on } U_0$.
Therefore, $J(u)=1$ on $U_0$. Since $\phi(1)=0$ as proved in Proposition \ref{prpnzh1}, we obtain the boundary condition that $u=0$ on $\Sigma$. The latter part of Proposition \ref{prpn3d5} follows from the first part and Graham's work (see the first paragraph of this section).
\end{proof}

\section{Proof of Theorem \ref{thm ke metric}}\label{Sec proof of thm ke metric}
We establish the following proposition before proceeding to the proof of Theorem \ref{thm ke metric}.
\begin{prop}\label{prpnzh2}
	Let $n\geq 1$, $k\geq 1$ and $m=n+k$.  For given real numbers $\lambda_1 \leq \cdots \leq \lambda_n <1$, set
	\begin{equation*}
		P(y)=\bigl(y-\frac{2k}{m+1}\bigr)^{k-1}\prod_{i=1}^n \bigl(y-\frac{2k \lambda_i}{m+1}\bigr).
	\end{equation*}
	Let $Q(y)$ be the polynomial satisfying $\frac{dQ}{dy}=(m+1)yP(y)$ and $Q(\frac{2k}{m+1})=0$ (thus $Q$ is a monic polynomial of degree $m+1$ and is uniquely determined). Suppose $\hat{P}$ and $\hat{Q}$ are polynomials defined by
	$$\hat{P}(x)=x^{m-1}P(x^{-1}), \quad \hat{Q}(x)=x^{m+1} Q(x^{-1}).$$
    Then the following conclusions hold:

\begin{itemize}		
\item[(1)] There exists a unique real analytic function $Z=Z(r)$ on $[-1, 1]$ (meaning it extends real analytically to some open interval containing $[-1, 1]$) satisfying the following conditions:
	\begin{equation}\label{eqnzodenz}
		rZ'\hat{P}(Z)+\hat{Q}(Z)=0, \quad Z(1)=0.
	\end{equation}
	Moreover, $Z$ is an even function satisfying $Z(0)=\frac{m+1}{2k}$, $Z'(0)=0$, $Z'<0$ on $(0, 1]$ and $Z'(1)=-1$, $Z''(0)<0.$ Consequently, $Z \in (0, \frac{m+1}{2k})$ on $(-1, 0) \cup (0, 1).$

\item[(2)] Let $\phi(r)=2(\frac{r^{2k-1}}{-Z'\hat{P}(Z)})^{\frac{1}{m+1}}\, Z$.  Then $\phi$ is real analytic on $[-1, 1]$. In addition, $\phi$ is an even function satisfying $\phi>0$ on $(-1, 1)$ and $\phi(1)=0.$ Moreover, $\phi$ satisfies $(m+1)rZ\phi'+(m+1-2kZ)\phi=0$ and $\phi'(1)=-2.$
\end{itemize}
\end{prop}

\begin{rmk}
	For the polynomials $P, Q, \hat{P}$ and $\hat{Q}$ defined in Proposition \ref{prpnzh2}, note that they satisfy $\hat{P}(0)=\hat{Q}(0)=1$, $P>0$ on $(\frac{2k}{m+1}, +\infty)$, $\hat{P}>0$ on $(0, \frac{m+1}{2k})$ and $\frac{d Q}{dy}>0$ on $(\frac{2k}{m+1}, \infty)$, $Q>0$ on $(\frac{2k}{m+1}, \infty)$, $\hat{Q}>0$ on $(0, \frac{m+1}{2k})$ and $\hat{Q}(\frac{m+1}{2k})=0$.
\end{rmk}

\begin{proof}
	The proposition was proved in \cite[Proposition 2.7]{EXX22} for $k=1$. We will extend the ideas in \cite{EXX22} to prove for the case $k\geq 2$.  Writing
\begin{equation}\label{constant lambda}
	\lambda=\frac{m+1}{2k},
\end{equation}
we can express $\hat P$ as follows
	\begin{equation*}
		\hat{P}(x)=\bigl(1-\frac{x}{\lambda}\bigr)^{k-1}\prod_{i=1}^n \bigl(1-\frac{\lambda_i }{\lambda}x \bigr):=(x-\lambda)^{k-1} h(x).
	\end{equation*}
	Here $h(x)=(-1)^{k-1}\lambda^{-(k-1)}\prod_{i=1}^n \bigl(1-\frac{\lambda_i}{\lambda}x\bigr)$ and it is a polynomial satisfying $h(\lambda)\neq 0$. Note the polynomial $\hat{Q}$ has a zero of order $k$ at $x=\lambda$. Thus we can write $\hat{Q}(x)=(x-\lambda)^k g(x)$ for some polynomial $g$ satisfying $g(\lambda)\neq 0$.

	We introduce the following lemma on the polynomials $h$ and $g$. For two functions $f_1$ and $f_2$, we write $f_1\sim f_2$ as $x\rightarrow \lambda$ if $\lim_{x\rightarrow \lambda} \frac{f_1(x)}{f_2(x)}=1$.
	\begin{lemma}\label{lemma g and h}
		It holds that $g(\lambda)=-2h(\lambda)$ and $(-1)^{k-1}h(\lambda)>0$. In particular, $\frac{\hat{Q}(x)}{\hat{P}(x)}\sim -2 (x-\lambda)$ as $x\rightarrow \lambda$.
	\end{lemma}

	\begin{proof}
		Note that $(-1)^{k-1}h(\lambda)=\lambda^{-(k-1)}\prod_{i=1}^n \bigl(1-\lambda_i\bigr)$, which is clearly positive as all $\lambda_i$'s are strictly less than $1$. We next prove the first identity. For that, we notice that
		\begin{equation*}
			h(\lambda)=\frac{1}{(k-1)!}\frac{d^{k-1} \hat{P}(x)}{dx^{k-1}}\Big|_{x=\lambda} \quad \mbox{ and } \quad g(\lambda)=\frac{1}{k!} \frac{d^k\hat{Q}(x)}{d x^k}\Big|_{x=\lambda}.
		\end{equation*}
		Since $\hat{Q}(x)=x^{m+1}Q(\frac{1}{x})$, by the definition of polynomial $Q$ we have
		\begin{equation*}
			\frac{d\hat{Q}(x)}{dx}=(m+1)x^mQ\bigl(\frac{1}{x}\bigr)-x^{m-1}Q'\bigl(\frac{1}{x}\bigr)=\frac{(m+1)}{x}\hat{Q}(x)-\frac{(m+1)}{x}\hat{P}(x).
		\end{equation*}
		As $\hat{P}$ and $\hat{Q}$ respectively have a zero of order $(k-1)$ and $k$ at $x=\lambda$, if we further differentiate the above equation $(k-1)$ times and evaluate it at $x=\lambda$, then
		\begin{equation*}
			\frac{d^k\hat{Q}(x)}{dx^k}\Big|_{x=\lambda}=-\frac{(m+1)}{\lambda}\frac{d^{k-1}\hat{P}(x)}{dx^{k-1}}\Big|_{x=\lambda}=-2k!\,h(\lambda).
		\end{equation*}
		It follows that $g(\lambda)=-2h(\lambda)$. As a result, we get
		\begin{equation*}
			\frac{\hat{Q}(x)}{\hat{P}(x)}=\frac{g(x)}{h(x)} (x-\lambda)\sim -2(x-\lambda) \quad \mbox{ as } x\rightarrow \lambda.
		\end{equation*}
		So the proof is completed.
	\end{proof}

	Now we are ready to prove part (1) of Proposition \ref{prpnzh2}. It follows easily from the assumption and elementary ODE theory that the ODE in (\ref{eqnzodenz}) has a real analytic solution $Z$ in some open interval $I$ containing $1$. Since $\hat{P}(0)=1$ and  $\hat{Q}(0)=1$, we see the ODE in (\ref{eqnzodenz}) implies $Z'(1)=-1$. Set
\begin{equation}\label{eq:t0}t_0=\inf \{t \in [0, 1): ~\text{on}~(t,1] \text{ there exists a real analytic solution}~Z~\text{to}~(\ref{eqnzodenz})~\text{with}~Z'<0\}.
\end{equation}
	By definition, $0 \leq t_0 <1$ and  on $(t_0,1]$ there is a real analytic solution $Z$ to (\ref{eqnzodenz})  with $Z'<0$.
	
	\begin{lemma} \label{lm1}
		The number defined in \eqref{eq:t0} satisfies $t_0=0.$ Consequently, on $(0, 1]$ there exists a (unique) real analytic solution $Z$ to \eqref{eqnzodenz} and it satisfies $Z'<0.$
	\end{lemma}
\begin{proof} Seeking a contradiction, we assume $t_0 >0.$ Since $Z$ is decreasing on $(t_0, 1),$ we conclude that $\mu:=\lim_{r \rightarrow t_0^+} Z(r)>0$ exists (allowing {\em a priori} $\mu=+\infty$ as the limit). We note that $\mu \leq \lambda$, where $\lambda$ is the number introduced in \eqref{constant lambda}. For, if this were not the case, then since $Z(1)=0$ there would exist some $t^* \in (t_0, 1)$ such that $Z(t^*)=\lambda$. By Lemma
	\ref{lemma g and h}, at $r=t^*$ the ODE \eqref{eqnzodenz} gives $Z'(t^*)=0$, contradicting the fact $Z'(t^*)<0$. Therefore, we proceed by examining the following two cases.
	
\textbf{Case I.} Assume $\mu=\lambda$.
In this case, by Lemma \ref{lemma g and h} the ODE \eqref{eqnzodenz} gives
	\begin{equation*}
		-Z'(r)=\frac{\hat{Q}(Z)}{r\hat{P}(Z)}\sim -2\frac{Z-\lambda}{r} \quad \mbox{ as } r \rightarrow t_0^+.
	\end{equation*}
	Thus, there exists some constants $\delta>0$ and $C>0$ such that
$-Z'(r)\leq C(\lambda-Z)$ for any $r\in (t_0, t_0+\delta).$
That is, $(\log(\lambda-Z))'\leq C$. By taking the integral we obtain
	\begin{equation*}
		\log\bigl(\lambda-Z(r)\bigr)\big|_{r=t_0}^t\leq C(t-t_0) \quad \mbox{ for any } t\in (t_0, t_0+\delta).
	\end{equation*}
	But this is impossible as the left hand side is $+\infty$ while the right hand side is a finite number.
	
\textbf{Case II.} Assume $\mu<\lambda$. In this case, $\hat{P}(\mu)>0$ and thus $\frac{\hat{Q}(Z)}{r\hat{P}(Z)}$ is a smooth function at $(r, Z)=(t_0, \mu)$. Therefore, the following initial value problem has a real analytic solution $\widetilde{Z}$ on some open interval $J$ containing $t_0$:
	$$rZ'\hat{P}(Z)+\hat{Q}(Z)=0, \quad Z(t_0)=\mu.$$
	Shrinking $J$ if necessary, we can assume $\widetilde{Z}'<0$ on $J$. By the uniqueness of solutions in the ODE theory, we can glue the previous solution
	with $\widetilde{Z}$ to obtain a real analytic solution to (\ref{eqnzodenz}), still called $Z$, on some open interval containing $[t_0, 1],$ which still satisfies $Z'<0.$ This contradicts the definition of $t_0$.
	
Since in each case there is a contradiction, we must have $t_0=0$ and this proves Lemma \ref{lm1}.
\end{proof}
	
	\smallskip
	
	By Lemma \ref{lm1}, $Z$ is decreasing on $(0, 1)$ and therefore $\mu=\lim_{r \rightarrow 0^+} Z(r)>0$ exists. By the same reasoning as in the proof of Lemma \ref{lm1}, we must have $\mu \leq \lambda=\frac{m+1}{2k}$. In fact, it holds that $\mu=\lambda$. Assume $\mu < \lambda.$ Note $\hat{P}, \hat{Q}>0$ on $[0, \mu]$. Since $-Z'=\frac{\hat{Q}}{r\hat{P}}$, we have $-Z' \geq \frac{c}{r}$ on $(0, 1)$ for some positive constant $c$. This contradicts the fact that $Z$ is bounded on $(0, 1).$  Hence we must have $\mu = \lambda$, i.e., $\lim_{r \rightarrow 0^+} Z(r)=\lambda.$ Thus, $Z$ is decreasing from $\lambda$ to $0$ on $[0, 1]$.
	
	We write $Z(r)=\lambda+rG(r)$ for some real analytic function $G$ on $(0, 1]$. It is clear that $G<0$ on $(0,1]$ as $Z$ is decreasing from $\lambda$ to $0$ on $[0, 1]$. We have the following lemma on $G$.
	\begin{lemma}\label{lm2}
		It holds that $\lim_{r \rightarrow 0^+} G(r)=0.$
	\end{lemma}
	\begin{proof}
		Note
		\begin{align*}
			\hat{P}(Z)=&(Z-\lambda)^{k-1} h(Z)=r^{k-1}G^{k-1} h(Z),\\
			\hat{Q}(Z)=&(Z-\lambda)^k g(Z)=r^kG^k g(Z).
		\end{align*}
		By substituting these identities together with $Z(r)=\lambda+rG(r)$ into the ODE in \eqref{eqnzodenz}, we obtain
		\begin{equation*}
			(rG)'h(Z)+Gg(Z)=0 \quad \mbox{ on } (0,1).
		\end{equation*}
		We can rewrite it into
		\begin{equation*}
			\frac{G'}{G}=-\frac{h(Z)+g(Z)}{rh(Z)} \quad \mbox{ on } (0,1).
		\end{equation*}
		By Lemma \ref{lemma g and h}, $-\frac{h(Z)+g(Z)}{h(Z)}\rightarrow 1$ as $r\rightarrow 0^+$. Consequently, there exists some constant $\delta>0$ such that
		\begin{equation*}
			\frac{G'}{G}>\frac{1}{2r} \quad \mbox{ for any } r\in (0, \delta).
		\end{equation*}
		As a result, $\frac{-G}{\sqrt{r}}$ is increasing on $(0, \delta)$, which in particular implies that $\frac{-G}{\sqrt{r}}$ is bounded from above on $(0, \delta)$. Therefore, $\lim_{r\rightarrow 0^+} G(r)=0$.
	\end{proof}
	
We may now further write $Z$ as $Z(r)=\lambda+r^2 W(r)$ for some real analytic function $W$ on $(0, 1]$. Clearly, $W=rG<0$ on $(0, 1]$, and $W(1)=-\lambda$. In addition, we have the following.
	\begin{lemma}\label{lm3}
		The limit $\lim_{r\rightarrow 0^+} W(r)$ exists and it is a negative number.
	\end{lemma}
	\begin{proof}
	We first note that
		\begin{align*}
			\hat{P}(Z)=&(Z-\lambda)^{k-1} h(Z)=r^{2k-2}W^{k-1} h(Z),\\
			\hat{Q}(Z)=&(Z-\lambda)^k g(Z)=r^{2k}W^k g(Z).
		\end{align*}
	Combining this with $Z=\lambda+r^2 W$ and the ODE in \eqref{eqnzodenz}, we obtain
	\begin{equation*}
		r(r^2 W)' h(Z)+(r^2W)g(Z)=0 \quad \mbox{ for } r\in (0,1).
	\end{equation*}
	As $h$ is nonvanishing on $[0, \lambda]$, we can simplify the above equation into
	\begin{equation}\label{eqnzw}
		W'=-\frac{2h(Z)+g(Z)}{r\,h(Z)} W \quad \mbox{ for } r\in (0,1).
	\end{equation}
	Recalling $Z=\lambda+rG$, we have $2h(Z)+g(Z)=2h(\lambda+rG)+g(\lambda+rG)$. As $h$ and $g$ are both polynomials, $2h(\lambda+rG)+g(\lambda+rG)$ is a polynomial in $rG$. Moreover, the constant term equals $2h(\lambda)+g(\lambda)=0$ by Lemma \ref{lemma g and h}. By Lemma \ref{lm2}, we deduce that
$$
f(r):=\frac{2h(\lambda + rG)+g(\lambda + rG)}{rh(\lambda+rG)}
$$
extends to a continuous function on $[0, 1].$ 
	Then using (\ref{eqnzw}) we obtain
	$$\ln (-W(r))=\ln (-W(1)) +\int_{r}^1 f(t)dt=\ln \lambda +\int_{r}^1 f(t)dt \quad \mbox{ for any } r \in (0, 1) $$
	Consequently, $\lim_{r \rightarrow 0^+} W(r)=-\lambda\exp{(\int_0^1 f(t) dt)}$, which is a negative real number.
	\end{proof}

Now $W$ naturally extends to a continuous function on $[0, 1]$. Set $a=W(0)=\lim_{r \rightarrow 0^+} W(r)$. We have the following lemma.
\begin{lemma}\label{lm4}
	There exists a unique real analytic function $T_0(r)$ at $r=0$ satisfying the following initial value problem:
	\begin{equation}\label{eqntp}
		T'=-\frac{2h(\lambda + r^2 T)+g(\lambda + r^2 T)}{r\,h(\lambda + r^2 T)}T, \quad T(0)=a.
	\end{equation}
	Moreover, the function is even on $(-\epsilon, \epsilon)$ for some small $\epsilon>0$.
\end{lemma}

\begin{rmk}\label{rmk16}
	Note $2h(\lambda + r^2 T)+g(\lambda + r^2 T)$ is a polynomial in $r^2T$, whose constant term equals $2h(\lambda)+g(\lambda)=0$ by
	Lemma \ref{lemma g and h}. Therefore $(2h(\lambda + r^2 T)+g(\lambda + r^2 T))/r$ is a polynomial in $r$ and $T$.
\end{rmk}

\begin{proof}[Proof of Lemma $\ref{lm4}$] By Remark \ref{rmk16}, the right hand side of the ODE in (\ref{eqntp}) is real analytic in a neighborhood of $(r, T)=(0, a)$. Therefore the existence and uniqueness of the solution, as well as its real analyticity, follow from elementary ODE theory. Note if $T_0$ is a solution  to the initial value problem \eqref{eqntp}, then so is $T_0(-r)$. By uniqueness of the solution, $T_0$ is an even function.
\end{proof}

Let $T_0: (-\epsilon, \epsilon) \rightarrow \mathbb{R}$ be as in Lemma \ref{lm4} and recall the function $W$ defined before Lemma \ref{lm3}. 
Note $T_0$ and $W$ are both functions in $C^{\omega}(0,\epsilon)\cap C[0, \epsilon)$ satisfying the following ODE:
\begin{equation}\label{eqntwr}
	T'=-\frac{2h(\lambda + r^2 T)+g(\lambda + r^2 T)}{rh(\lambda + r^2 T)}T~\text{on}~(0, \epsilon), \quad T(0)=a.
\end{equation}
By basic ODE theory (cf. the proof of Lemma 2.13 in \cite{EXX22}), it follows that $W=T_0$ on $[0, \epsilon)$.



We now continue the proof of Proposition \ref{prpnzh2}. Let $\lambda, T_0$ be as above and set $\Psi=\lambda + r^2 T_0$. Then $\Psi$ is a real analytic even function on $(-\epsilon, \epsilon).$ Moreover, $\Psi=Z$ on $[0, \epsilon)$.
Therefore we can glue $Z$ with $\Psi$, and then apply the even extension to obtain a real analytic function on $[-1, 1]$, which we still denote by $Z$. It is clear that this new function $Z$ still satisfies the ODE in (\ref{eqnzodenz}). Moreover, $Z''(0)=2W(0)=2a<0$. This proves part (1) of Proposition \ref{prpnzh2}. 

We next prove part (2) of Proposition \ref{prpnzh2}. Recall $Z(r)=\lambda+r^2W(r)$ and $\hat{P}(Z)=(Z-\lambda)^{k-1}h(Z)=r^{2k-2}W^{k-1}h(Z)$. It follows that
\begin{equation*}
	\frac{r^{2k-1}}{-Z'\hat{P}(Z)}=\frac{r^{2k-1}}{-(2rW+r^2W')r^{2k-2}W^{k-1}h(Z)}=\frac{1}{-(2W+rW')W^{k-1} h(Z)}.
\end{equation*}
At $r=0$, $-(2W+rW')W^{k-1} h(Z)=-2W^k(0)\, h(\lambda)>0$ since $W(0)<0$ by Lemma \ref{lm3} and $(-1)^{k-1}h(\lambda)>0$ by  Lemma \ref{lemma g and h}. Hence $\frac{r^{2k-1}}{-Z'\hat{P}(Z)}$ is real analytic at $r=0$. By the definition of $\phi$ and the properties of $Z$ in part (1), $\phi$ is real analytic and even on $[-1, 1]$. It is also clear that $\phi>0$ on $(-1, 1)$ and $\phi(1)=0$. The latter assertion in part (2) can be proved identically as in Proposition \ref{prpnzh1}. This finishes the proof of Proposition \ref{prpnzh2}.	
\end{proof}

We are now ready to prove Theorem \ref{thm ke metric}.
\begin{proof}[Proof of Theorem $\ref{thm ke metric}$]
	Choose a coordinates chart $(D, z)$ of $M$ together with a local frame $\{e_{\alpha}\}_{\alpha=1}^k$ of $E$ over $D$. Writing $\pi: E \rightarrow M$ for the canonical fiber projection, we have $$\pi^{-1}(D)=\Big\{\sum_{\alpha=1}^k\xi_{\alpha} e_{\alpha}(z): (z, \xi) \in D \times \mathbb{C}^k \Big\}.$$
Under this trivialization, the ball bundle $B(E)$ and the sphere bundle $S(E)$ over $D$ can be expressed as follows:
\begin{align*}
	B(E) \cap \pi^{-1}(D) =&\Big\{w=(z, \xi) \in D \times \mathbb{C}^k: \sum_{\alpha, \beta=1}^k h_{\alpha\bar{\beta}}(z, \Ol{z}) \xi_{\alpha}\oo{\xi_{\beta}}<1 \Big\},\\
	S(E) \cap \pi^{-1}(D) =&\Big\{w=(z, \xi) \in D \times \mathbb{C}^k: \sum_{\alpha, \beta=1}^k h_{\alpha\bar{\beta}}(z, \Ol{z}) \xi_{\alpha}\oo{\xi_{\beta}}=1 \Big\}.
\end{align*}
Here $h_{\alpha\bar{\beta}}(z)=h_{\alpha\bar{\beta}}(z, \Ol{z})=h(e_{\alpha}, e_{\beta})$. In the local coordinates, we write $g=\sum_{i, j=1}^n g_{i\Ol{j}} dz_i\otimes d\oo{z_j}$. As $g$ is induced by $-\Ric(E, h)$, we have $g_{i\bar{j}}=\frac{\partial^2 \log H}{\partial z_i \partial \oo{z_j}}$ where $H=\det(h_{\alpha\bar{\beta}})$. Let $G(z)=G(z, \Ol{z})=\det (g_{i\Ol{j}})>0$. We let $\lambda_1\leq \cdots \leq \lambda_n<1$ be the Ricci eigenvalues of $(M, g)$ and $\phi$ be the function resulting from Proposition \ref{prpnzh2}. In the local coordinates, the K\"ahler form $\widetilde{\omega}$ in Theorem \ref{thm ke metric} is given by $\widetilde{\omega}=-i \partial \Ol{\partial} \log u$, where $u(w)=k^{\frac{n}{m+1}}(GH)^{-\frac{1}{m+1}}\phi(|w|_h)$. 
Since $\phi$ is real analytic and even on $[-1,1]$, $u$ is smooth in a neighborhood of $\Ol{B(E)} \cap \pi^{-1}(D)$. Consequently, $\widetilde{\omega}$ is a smooth \k form on $\oo{B(E)}\cap \pi^{-1}(D)$.
By repeating the proof of Proposition \ref{prpn3d5} (and the smoothness of $u$), it follows that 
$u=0~\text{on}~S(E) \cap \pi^{-1}(D)$ and
$J(u)=1~\text{on}~B(E) \cap \pi^{-1}(D).$

Since $J(u)=1$, or equivalently, $\det \big((-\log u)_{s\Ol{t}}\big)_{1 \leq s, t \leq m}=u^{-(m+1)}$ in $B(E) \cap \pi^{-1}(D),$ and $u$ is a local defining function of some strongly pseudoconvex piece of the boundary, we conclude that $\widetilde{\omega}$ is positive definite in $B(E) \cap \pi^{-1}(D)$. 
Also $J(u)=1$ implies that the metric $\widetilde{g}$ induced by $\widetilde{\omega}$ has constant Ricci curvature $-(m+1).$
Since the coordinates chart $D$ is arbitrarily chosen, $\widetilde{g}$ is a \ke metric in $B(E)$.

It remains to prove that the metric $\widetilde{g}$ is complete on $B(E)$. By the Hopf-Rinow Theorem, it suffices to show $(B(E), \widetilde{g})$ is geodesically complete.
Let $\gamma: [0,a) \rightarrow B(E)$ be a  non-extendible geodesic in $B(E)$ of unit speed with respect to $\widetilde{g}$. We only need to show that $a=+\infty$, that is, $\gamma$ has infinite length. For that, we first establish the following lemma.

\begin{lemma}\label{lm3d15}
	The metric $\widetilde{g}$ satisfies
	\begin{equation}\label{lm3d15 eq}
		\widetilde{g}\geq -\frac{1}{m+1} \pi^*(\Ric)+\frac{1}{m+1}\pi^*(g).
	\end{equation}
	Consequently, $\widetilde{g} \geq \frac{1-\lambda_n}{m+1} \pi^{*}(g)$ in $B(E)$.
\end{lemma}

\begin{proof}
	Since the validity of \eqref{lm3d15 eq} is independent of the choice of local coordinates chart of $B(E)$, it suffices to establish \eqref{lm3d15 eq} in an arbitrary coordinate chart. Given $p\in B(E)$, recall that in the proof of Lemma \ref{lemma compute the Hessian det}, we have proved that there exist some local coordinates $w=(z, \xi)$, in which $z(\pi(p))=0$ and the metric $\widetilde{g}_{s\bar{t}}$ at $w=(0, \xi)$ with $0<|w|_h<1$ can be expressed as
	\begin{equation}\label{metric on the bundle eq}
		\widetilde{g}_{s\bar{t}}=
		\begin{dcases}
			-\frac{1}{m+1}R_{i\bar{j}}(0)+\frac{Y}{2k}g_{i\bar{j}}(0) & \mbox{ for } 1\leq s, t\leq n, \\
			\frac{1}{2|\xi|^2}\bigl(Y-\frac{2k}{m+1}\bigr)\delta_{\alpha\beta}+\bigl(\frac{XY'}{4}-\frac{Y}{2}+\frac{k}{m+1}\bigr)\frac{\oo{\xi_{\alpha}}\xi_{\beta}}{|\xi|^4} & \mbox{ for } n+1\leq s, t\leq n+k, \\
			0 & \mbox{ otherwise, }
		\end{dcases}
	\end{equation}
	where $i=s, j=t$ when $1\leq s, t\leq n$ and $\alpha=s-n, \beta=t-n$ when $n+1\leq s, t\leq n+k$. Recall $X=X(w)=(\sum_{\alpha, \beta=1}^n h_{\alpha\bar{\beta}}\xi_{\alpha}\oo{\xi_{\beta}})^{1/2}$, $Y=Y(w)=\frac{2k}{m+1}-X\frac{\phi'(X)}{\phi(X)}$ and $\phi$ is defined in Proposition \ref{prpnzh2}. As in \S\ref{Sec proof of thm obstruction flat}, by \eqref{eqn3d5} we have $Y=y(r)|_{r=X}=\frac{1}{Z}|_{r=X}$. Furthermore, by Proposition \ref{prpnzh2} it follows that $Y(w)\geq \frac{2k}{m+1}$ for any $w\in B(E)$. Note that $(\widetilde{g}_{s\bar{t}})_{1\leq s, t\leq n+k}$ is a block diagonal matrix. As it is positive definite, the diagonal blocks, $(\widetilde{g}_{st})_{1\leq s, t\leq n}$ and $(\widetilde{g}_{st})_{n+1\leq s, t\leq n+k}$, are also individually positive definite. As a result,
	\begin{equation*}
		\widetilde{g}\geq \sum_{i, j=1}^n\widetilde{g}_{i\bar{j}} dz_i\otimes d\oo{z_j}\geq -\frac{1}{m+1} \pi^*(\Ric)+\frac{1}{m+1}\pi^*(g) ~\mbox{ at any $w=(0, \xi)$ with $0<|w|_h<1$}.
	\end{equation*}
	By continuity, the above actually holds at any $w=(0, \xi)$ with $|w|_h<1$, and thus \eqref{lm3d15 eq} is proved. Finally, by the assumption on the Ricci eigenvalues, we have $R_{i\bar{j}}\leq \lambda_n g_{i\bar{j}}$ with $\lambda_n<1$. The latter part of the lemma follows easily.
\end{proof}

With Lemma \ref{lm3d15}, the remaining part of the proof is identical to that of Theorem 1.4 in \cite{EXX22}. We omit the details.
\end{proof}

\section{Proofs of Corollaries}\label{Sec an interesting case}
In this section, we consider the case where the base manifold $M$ is a domain $D$ in $\mathbb{C}^n$, and prove Corollary \ref{cor CY solution} and \ref{cor homogeneous domain}, as well as Proposition \ref{prop egg domain}. We also exhibit some explicit examples as applications.

We first prove Corollary \ref{cor CY solution}.
\begin{proof}[Proof of Corollary $\ref{cor CY solution}$]
	Let $L=D\times \mathbb{C}$ be the trivial line bundle over $D$. By the assumption on $h$, the Hermitian line bundle $(L, h)$, 
	is negative. Take the Hermitian vector bundle $(E, h_E)$ as $(L, h)\oplus \cdots \oplus (L, h)$, the direct sum of $k$ copies of $(L, h)$. Then the ball bundle $B(E)$ and the sphere bundle $S(E)$ of $(E, h_E)$ are respectively given by
	\begin{align*}
		B(E)=&\bigl\{ w=(z, \xi)\in D\times \mathbb{C}^k: |\xi|^2h(z,\bar{z})-1<0 \bigr\},\\
		S(E)=&\bigl\{ w=(z, \xi)\in D\times \mathbb{C}^k: |\xi|^2h(z,\bar{z})-1=0 \bigr\}.
	\end{align*}
	Note that $B(E)$ coincides with the domain $\Omega\subset \mathbb{C}^m$ (recall $m=n+k$) as defined in \eqref{ball bundle for domain case}, and $S(E)$ coincides with the hypersurface $\Sigma$ in $\mathbb{C}^m$ defined by \eqref{sphere bundle for domain case}. In addition, the Hermitian vector bundle $(E, h_E)$ is curvature split by Proposition \ref{curvature split when the bundle splits prop} and it is Griffiths negative by Remark \ref{Griffiths negative if and only if each line bundle is negative rmk}. Moreover, since $\Ric(E, h_E)=k\Ric(L, h)$, the \k metric induced by $-\Ric(E, h_E)$ is the metric $g$ given in the assumption. 
	By Theorem \ref{thm ke metric}, the unique complete \ke metric $\widetilde{g}$ with Ricci curvature $-(m+1)$ is given by \eqref{thm ke metric eq}. The explicit formula of the Cheng\textendash Yau\textendash Mok solution $u$, defined as $(\det(\widetilde{g}_{s\bar{t}}))^{-\frac{1}{m+1}}$, can be seen from the proof of Theorem \ref{thm ke metric}. Since the function $\phi$ is real analytic on $[-1, 1]$ by Proposition \ref{prpnzh2} and $G, H$ are both real analytic on $D$ by the assumption, the Cheng\textendash Yau\textendash Mok solution $u$ extends real analytically across $\Sigma$.
\end{proof}


We next prove Corollary \ref{cor homogeneous domain}.

\begin{proof}[Proof of Corollary $\ref{cor homogeneous domain}$]
	We first note that $g$ is complete since $(D, g)$ is homogeneous. (For the proof of this fact, see for example \cite{Lu}.) Moreover, the homogeneity of $(D, g)$ also implies that $g$ has constant Ricci eigenvalues. The result is now a direct consequence of Corollary \ref{cor CY solution} and the following:

\noindent
{\bf Claim.} The Ricci eigenvalues of $g$ are all negative (and thus in particular strictly less than $1$).
	\begin{proof}[Proof of Claim]
		Let $g_B$ be the Bergman metric of $D$. We denote by $\Vol(g)$ and $\Vol(g_B)$ the volume forms of $g$ and $g_B$ respectively, which are $(n, n)$ forms on $D$. Set $\Phi:=\Vol(g)/\Vol(g_B)$, which is a well-defined function on $D$. Let $\Iso(g)$ be the group of holomorphic isometries of $(\Omega, g)$. Since every biholomorphism also preserves $g_B$, the group $\Iso(g)$ actually preserves both $g$ and $g_B$. Thus, $\Phi$ is invariant under the action of $\Iso(g)$. As $\Iso(g)$ acts transitively on $\Omega$, $\Phi$ is constant on $\Omega$, that is, $\Vol(g)$ and $\Vol(g_B)$ are the same up to some (positive) constant factor. Therefore, $g$ and $g_B$ have the same Ricci form. By the fact $\Ric(g_B)=-g_B$ (see \cite{Ber70} for example), we get $\Ric(g)\cdot g^{-1}=\Ric(g_B)\cdot g^{-1}=-g_B\cdot g^{-1}$. So all the Ricci eigenvalues of $g$ are negative and the proof is completed.
	\end{proof}
\end{proof}

We now present some examples as applications of the above corollaries.
\begin{ex}
	Let $D$ be a bounded homogeneous domain in $\mathbb{C}^n$. Write $K_{D}(z, z)$ for its Bergman kernel and $g_B$ for the Bergman metric. Since $g_B$ is biholomorphic invariant, the manifold $(\Omega, g_B)$ is homogeneous K\"ahler. Given $\lambda\in \mathbb{R}^+$ and $k\in \mathbb{Z}^+$, we set $h=(K_D)^{\lambda}$ and consider the domain $\Omega\subset D\times\mathbb{C}^k$ and the hypersurface $\Sigma\subset
	D \times \mathbb{C}^k$ as defined in \eqref{ball bundle for domain case} and \eqref{sphere bundle for domain case}. By Corollary \ref{cor homogeneous domain}, the Cheng\textendash Yau\textendash Mok solution of $\Omega$ is given by the following with $m=n+k$:
	\begin{equation*}
		u(w)=k^{\frac{n}{m+1}} (GH)^{-\frac{1}{m+1}} \phi(|\xi|h^{\frac{1}{2}}),
	\end{equation*}
	where $H=(K_D)^{k \lambda}$, $G=k^n\lambda^n G_B$ and $G_B$ is the volume density of $g_B$. Moreover, the boundary hypersurface $\Sigma$ is obstruction flat and $u$ extends real analytically across $\Sigma$.
\end{ex}

\begin{ex}\label{ex bounded domain of holomorphy}
	Let $D$ be a bounded domain of holomorphy in $\mathbb{C}^n$. Suppose $g_0=\bigl((g_0)_{i\bar{j}}\bigr)$ is the complete \ke metric with negative Ricci curvature $\lambda_0$. (The existence and uniqueness of such a metric is guaranteed by the work of Mok\textendash Yau \cite{MoYau}.) Let $h$ be a real analytic function on $D$ such that $(g_0)_{i\bar{j}}=\frac{\partial^2 \log h}{\partial z_i \partial \oo{z_j}}$. (One particular choice of such an $h$ is $\bigl(\det((g_0)_{i\bar{j}})\bigr)^{-1/\lambda_0}$ as $g_0$ is K\"ahler-Einstein.) For a given $k\in \mathbb{Z}^+$, consider the domain $\Omega\subset D\times\mathbb{C}^k$ and the hypersurface $\Sigma\subset
	D \times \mathbb{C}^k$ as defined in \eqref{ball bundle for domain case} and \eqref{sphere bundle for domain case}. By Corollary \ref{cor CY solution}, the Cheng\textendash Yau\textendash Mok solution of $\Omega$ is given by the following with $m=n+k$:
	\begin{equation*}
		u=k^{\frac{n}{m+1}} (GH)^{-\frac{1}{m+1}}\phi(|\xi|h^{\frac{1}{2}}).
	\end{equation*}
	where $H=h^k$ and $G=k^n\det\bigl((g_0)_{i\bar{j}}\bigr)$. Moreover, the boundary hypersurface $\Sigma$ is obstruction flat and $u$ extends real analytically across $\Sigma$.
	
	In particular, if we choose $h=1/{u_0}$ where $u_0$ is the Cheng\textendash Yau\textendash Mok solution for the domain $D$, then $g_0=\bigl((\log h)_{i\bar{j}}\bigr)$ is the complete \k metric with Ricci curvature $\lambda_0=-(n+1)$. A routine computation yields the following expression for the Cheng\textendash Yau\textendash Mok solution of $\Omega$,
	\begin{equation*}
		u=u_0\, \phi(|\xi|h^{\frac{1}{2}}).
	\end{equation*}
\end{ex}

\begin{ex}
	Given $l\geq 1$, for each $1\leq i\leq l$, let $D_i$ be a domain in $\mathbb{C}^{n_i}$, $g^i=\sum_{p, q=1}^{n_i} g_{p\bar{q}}^i dz_{p}^i\wedge d\oo{z_q^i}$ a complete \ke metric on $D_i$, and $h^i$ a real analytic function on $D_i$ such that $g^i$ is induced by $\sqrt{-1}\partial\dbar\log h^i$. Let $D=D_1\times \cdots \times
	D_l\subset \mathbb{C}^n$ with $n=n_1+\cdots+n_l$ and write $z=(z^1, \cdots, z^l)$ for each $z^i\in D_i$. Set $h(z,\bar{z})=\Pi_{i=1}^l h^i(z^i, \oo{z^i})$. For a fixed $k\in \mathbb{Z}^+$, consider the domain $\Omega\subset D\times\mathbb{C}^k$ and the hypersurface $\Sigma\subset
	D\times \mathbb{C}^k$ as defined in \eqref{ball bundle for domain case} and \eqref{sphere bundle for domain case}. Then the Cheng\textendash Yau\textendash Mok solution of $\Omega$ is given by the following with $m=n+k$:
	\begin{equation*}
		u(w)=k^{\frac{n}{m+1}} (GH)^{-\frac{1}{m+1}} \phi(|\xi|h^{\frac{1}{2}}),
	\end{equation*}
	where $H=h^k$ and $G=k^n\Pi_{i=1}^l\det(g^i_{pq})_{1\leq p, q\leq n_i}$. Moreover, the boundary hypersurface $\Sigma$ is obstruction flat and $u$ extends real analytically across $\Sigma$.
\end{ex}

To conclude the paper, we shall prove Proposition \ref{prop egg domain}. Before proceeding to the proof, we first consider the case of the ball bundle over a bounded domain of holomorphy $D$. As mentioned in Example \ref{ex bounded domain of holomorphy}, for a given negative real number $\lambda_0$, there exists a unique complete \ke metric such that $\Ric(g_0)=\lambda_0g_0$. As before, $L$ is the trivial line bundle over $D$ and $h$ is a Hermitian metric of $L$ such that $g_0$ is induced by $-c_1(L, h)$. The Hermitian vector bundle $(E, h_E)$ is the direct sum of $k$ copies of $(L, h)$. By Example \ref{ex bounded domain of holomorphy} the Cheng\textendash Yau\textendash Mok solution for $\Omega=B(E)$ as defined in \eqref{ball bundle for domain case} is
\begin{equation}\label{CYM solution for KE case}
	u=k^{\frac{n}{m+1}} (GH)^{-\frac{1}{m+1}}\phi(|\xi|h^{\frac{1}{2}}),~~\text{with}~m=n+k.
\end{equation}
where $H=h^k$, $G=k^n\det((g_0)_{i\bar{j}})$, and the function $\phi$ is given in Proposition \ref{prpnzh2} with the $\lambda_i$'s chosen as follows.
First, it is clear that $-c_1(E, h_E)$ is $-kc_1(L, h)$, which also induces a \ke metric $g=kg_0$ with negative constant Ricci curvature $\lambda=\lambda_0/k <0$. Write $\mu=\frac{2k\lambda}{m+1}$ and $\nu=\frac{2k}{m+1}$ and let all $\lambda_i$'s in Proposition \ref{prpnzh2} be $\lambda$. The polynomials $P(y)$ and $\hat{P}(x)$ are then given by
\begin{align} \label{P for the KE case}
	P(y)=(y-\nu)^{k-1} (y-\mu)^n, \qquad \hat{P}(x)=(1-\nu x)^{k-1} (1-\mu x)^n.
\end{align}
The polynomial $Q(y)$ from Proposition \ref{prpnzh2} satisfies the following properties:
\begin{lemma}\label{lemma Q for the KE case} The following hold:
\begin{itemize}
\item[(1)] The polynomial $Q$ is divisible by $(y-\nu)^k$. Moreover, there exists a polynomial $T(y)$ such that $Q(y)=(y-\mu)^{n+1}T(y)+c$, where $c$ is a real number satisfying $c=-(\nu-\mu)^{n+1} T(\nu)=Q(\mu)$.
\item[(2)] The number $c=0$ if and only if $\lambda=-\frac{n+1}{k}$ (i.e., $\lambda_0=-(n+1)$).
	In this case, $Q= (y-\mu)^{n+1} (y-\nu)^k$.
\end{itemize}
\end{lemma}

\begin{proof}
	We first prove part $(1)$. Recall by the definition in Proposition \ref{prpnzh2}, $Q$ satisfies
	\begin{equation*}
		\frac{d Q}{dy}=(m+1) y P(y) \quad \mbox{ and } \quad Q(\nu)=0.
	\end{equation*}
	It follows immediately that
	\begin{equation*}
		Q(y)=\int_{\nu}^y (m+1) tP(t) dt.
	\end{equation*}
	Note that we can write the integrand function as
	\begin{equation*}
		(m+1) y P(y)=(y-\nu)^{k-1} \sum_{j=0}^{n+1} a_j (y-\nu)^j \quad \mbox{ for some $a_j$'s in $\mathbb{R}$. }
	\end{equation*}
	We take the integration term by term and obtain
	\begin{equation*}
		Q(y)=\sum_{j=0}^{n+1} \frac{a_j}{k+j} (y-\nu)^{k+j}.
	\end{equation*}
	Therefore, $Q$ is divisible by $(y-\nu)^k$.
	
	To prove the latter assertion in part $(1)$, note that we can also write
	\begin{equation*}
		(m+1) y P(y)=(y-\mu)^n \sum_{j=0}^k b_j(y-\mu)^j \quad \mbox{ for some $b_j$'s in $\mathbb{R}$}.
	\end{equation*}
	As $Q(y)=\int_{\mu}^y (m+1) tP(t) dt+Q(\mu)$, we again take the integration term by term and obtain
	\begin{equation*}
		Q(y)=\sum_{j=0}^{k}\frac{b_j}{n+j+1} (y-\mu)^{n+j+1}+Q(\mu).
	\end{equation*}
	By setting
	\begin{equation*}
		T(y)=\sum_{j=0}^k \frac{b_j}{n+j+1} (y-\mu)^{j} \quad \mbox{and} \quad c=Q(\mu),
	\end{equation*}
	we have $Q(y)=(y-\mu)^{n+1} T(y)+c$. Since $Q(\nu)=0$, it follows that $c=-(\nu-\mu)^{n+1} T(\nu)$.
	
	Now we prove part (2). It is clear that $c=0$ if and only if $Q(\mu)=0$. Since $Q(\nu)=0$ in the assumption, the former is also equivalent to
	\begin{equation}\label{Q(mu)=0}
		\int_{\mu}^{\nu} \frac{d Q}{dy} dy=0, \quad \mbox{ i.e., } ~\int_{\mu}^{\nu} yP(y) dy=0.	
	\end{equation}
	By writing $y=\frac{\mu}{\mu-\nu}(y-\nu)-\frac{\nu}{\mu-\nu} (y-\mu)$, we have
	\begin{equation*}
		yP(y)=y(y-\nu)^{k-1} (y-\mu)^n=\frac{\mu}{\mu-\nu} (y-\nu)^k (y-\mu)^n-\frac{\nu}{\mu-\nu} (y-\nu)^{k-1} (y-\mu)^{n+1}.
	\end{equation*}
	Thus, \eqref{Q(mu)=0} is equivalent to
	\begin{equation*}
		\frac{\mu}{\mu-\nu} \int_{\mu}^{\nu}(y-\nu)^k (y-\mu)^n dy=\frac{\nu}{\mu-\nu} \int_{\mu}^{\nu}(y-\nu)^{k-1} (y-\mu)^{n+1}dy.
	\end{equation*}
	By setting $t=\frac{y-\mu}{\nu-\mu}$, it reduces to
	\begin{equation}\label{Q(mu)=0 2}
		\mu \int_0^1 (1-t)^k t^n dt=-\nu \int_0^1 (1-t)^{k-1} t^{n+1} dt.
	\end{equation}
	Recall the beta function is defined by
	\begin{equation*}
		B(p, q)= \int_0^1 t^{p-1} (1-t)^{q-1} dt
	\end{equation*}
	and $B(p, q)=\frac{(p-1)! (q-1)!}{(p+q-1)!}$. Therefore, \eqref{Q(mu)=0 2} writes into
	\begin{equation*}
		\mu \frac{k! n!}{(n+k+1)!} =-\nu \frac{(k-1)! (n+1)!}{(n+k+1)!}.
	\end{equation*}
	As $\mu=\lambda\nu$, we finally obtain that $c=0$ is equivalent to $\lambda=-\frac{n+1}{k}$.
	
	In this case, by part (1) we have both $(y-\nu)^k$ and $(y-\mu)^{n+1}$ divide $Q$. Since $Q$ is a monic polynomial of degree $n+k+1$, it follows that $Q(y)=(y-\nu)^k (y-\mu)^{n+1}$.
\end{proof}

In the special case that $P$ is given by \eqref{P for the KE case}, we will study the rationality of function $Z$ as defined in Proposition \ref{prpnzh2}.

\begin{prop}\label{prop equivalent conditions on rationality}
	Let $P$ and $\hat{P}$ be given by \eqref{P for the KE case} (and accordingly $Q$ and $\hat{Q}$ are both determined as in Proposition \ref{prpnzh2}). Let $Z$ and $\phi$ be as given in Proposition $\ref{prpnzh2}$. Then the following are equivalent.
	\begin{itemize}
		\item[(1)] $Z$ is rational.
		\item[(2)] $\phi^{m+1}$ is rational.
		\item[(3)] $\lambda=-\frac{n+1}{k}$.
	\end{itemize}

\end{prop}

\begin{rmk}
	Moreover,  when (1)-(3) in Proposition \ref{prop equivalent conditions on rationality} hold, we can see from the proof that $\mu=-\frac{2(n+1)}{m+1}$ and $Z=\frac{1-r^2}{2+\mu-\mu r^2}.$ Consequently, $\phi(r)=1-r^2$.
\end{rmk}

\begin{proof}[Proof of Proposition $\ref{prop equivalent conditions on rationality}$]	We first prove (1) is equivalent to (2). By Proposition \ref{prpnzh2}, we have
	\begin{equation*}
		\frac{1}{Z}=\frac{2k}{m+1}-r\frac{\phi'}{\phi}=\frac{2k}{m+1}-\frac{r}{m+1}\frac{(\phi^{m+1})'}{\phi^{m+1}}.
	\end{equation*}
Hence (2) implies (1). Conversely, recall that in Proposition \ref{prpnzh2} the function $\phi$ is defined by $		\phi=2\Bigl(\frac{r^{2k-1}}{-Z'\hat{P}(Z)}\Bigr)^{\frac{1}{m+1}} Z$. Since $\hat{P}$ is a polynomial, 
the rationality of $Z$ implies that of $\phi^{m+1}$.

Now it remains to show (1) is equivalent to (3). We shall first show that (3) implies (1). Suppose $\lambda=-\frac{n+1}{k}$. Then by Lemma \ref{lemma Q for the KE case} we get $Q(y)=(y-\mu)^{n+1}(y-\nu)^k$. Thus, $\hat{Q}(x)=(1-\mu x)^{n+1} (1-\nu x)^k$. Recall $\hat{P}=(1-\mu x)^{n}(1-\nu x)^{k-1}$ as given in \eqref{P for the KE case}. Since $Z$ satisfies $rZ'\hat{P}(Z)+\hat{Q}(Z)=0$ for $r\in [-1, 1]$, it follows that
	\begin{equation*}
		\frac{Z'}{(1-\mu Z)(1-\nu Z)}=-\frac{1}{r} \quad \mbox{ for } r\in (0, 1].
	\end{equation*}
	As $Z(1)=0$, by writing $\frac{1}{(1-\mu Z) (1-\nu Z)}=\frac{1}{\mu-\nu} \bigl(\frac{\mu}{1-\mu Z}-\frac{\nu}{1-\nu Z} \bigr)$ and integrating the above equation, we obtain
	\begin{equation*}
		\frac{1}{\mu-\nu}\Bigl(-\ln(1-\mu Z)+\ln(1-\nu Z) \Bigr)=-\ln r \quad \mbox{ for } r\in (0, 1].
	\end{equation*}
	Since $\mu-\nu=\frac{2k\lambda-2k}{m+1}=-2$, we get $\ln \frac{1-\nu Z}{1-\mu Z}=\ln (r^2)$, and further simplification yields
	\begin{equation*}
		Z=\frac{1-r^2}{\nu-\mu r^2}=\frac{1-r^2}{2+\mu-\mu r^2}.
	\end{equation*}
	It is clear that $Z$ is a rational function.
	
	Last we check that (1) implies (3). Suppose that $Z$ is rational. Recall $\hat{P}(x)=(1-\mu x)^{n}(1-\nu x)^{k-1}$. By Lemma \ref{lemma Q for the KE case}, we have $Q(y)=(y-\mu)^{n+1} T(y)+c$ where $T$ is some polynomial of degree $k$ and $c$ is some real number. By writing $T(y)=\sum_{j=0}^k a_j (y-\mu)^j$ for some $a_j\in \mathbb{R}$, we then have
	\begin{equation*}
		Q(y)=\sum_{j=0}^k a_j (y-\mu)^{n+1+j}+c \quad \mbox{ and } \quad \hat{Q}(x)=\sum_{j=0}^k a_j (1-\mu x)^{n+1+j} x^{k-j}+cx^{m+1}.
	\end{equation*}
Recall $rZ'\hat{P}(Z)+\hat{Q}(Z)=0$ for $r\in [-1, 1]$. We divide the equation by $(1-\mu Z)^{m+1}$ to obtain
	\begin{equation*}
		r\frac{Z'(1-\nu Z)^{k-1}}{(1-\mu Z)^{k+1}}+\sum_{j=0}^k a_j \frac{Z^{k-j}}{(1-\mu Z)^{k-j}}+c\frac{Z^{m+1}}{(1-\mu Z)^{m+1}}=0 \quad \mbox{ for } r\in [-1, 1].
	\end{equation*}
	Set $\eta(r)=\frac{Z(r)}{1-\mu Z(r)}$ for $r\in [-1, 1]$. Then $\eta'=\frac{Z'}{(1-\mu Z)^2}$ and $\frac{1-\nu Z}{1-\mu Z}=1+(\mu-\nu)\eta$. Therefore, we can rewrite the above equation into
	\begin{equation}\label{ODE on eta}
		r\eta'\bigl(1+(\mu-\nu)\eta\bigr)^{k-1}+\sum_{j=0}^k a_j \eta^{k-j}+ c\eta^{m+1}=0.
	\end{equation}	
	As $Z$ is rational, so is $\eta$. We can write $\eta=\frac{p}{q}$ for some coprime polynomials $p$ and $q$. Putting this into \eqref{ODE on eta} and multiplying the equation by $q^{m+1}$, we obtain
	\begin{equation*}
		r (p'q-pq') \bigl(q+(\mu-\nu)p\bigr)^{k-1} q^{n}+\sum_{j=0}^k a_j p^{k-j} q^{n+1+j}+cp^{m+1}=0
	\end{equation*}
	Assume $c\neq 0$. Then $p^{m+1}$ is divisible by $q$. As $p$ and $q$ are coprime, $q$ is a constant. Without losing of generality, we can assume $q=1$ and thus $\eta$ is just the polynomial $p$. Now that all terms in \eqref{ODE on eta} are polynomials, we can count their degrees. Note $c\eta^{m+1}$ is of degree $(m+1)\deg p$ while all other terms on the left hand side of \eqref{ODE on eta} are of degree less than or equal to $k \deg p$. Therefore, we have $\deg p=0$, that is, $p$ is a constant. So are the functions $\eta$ and $Z$. 
This is a contradiction as $Z$ is not constant by Proposition \ref{prpnzh2}. Hence we must have $c=0$. By Lemma \ref{lemma Q for the KE case}, (3) holds. So the proof is completed.
\end{proof}

We are now ready to prove Proposition \ref{prop egg domain}.
\begin{proof}[Proof of Proposition $\ref{prop egg domain}$]
	Let $D$ be the complex $n$-dimensional unit ball $\{z\in \mathbb{C}^n: |z|<1\}$. We introduce the function $h=(\frac{1}{1-|z|^2})^{1/p}$ and the metric $g_0=(g_0)_{i\bar{j}}=\bigl((\log h)_{i\bar{j}}\bigr)$. Note that $g_0$ is just $\frac{1}{p(n+1)}$ multiple of the Bergman metric of $D$. Thus $g_0$ is the complete \ke metric on $D$ with Ricci curvature equal to $\lambda_0=-p(n+1)$. If we take $g=k g_0$, then $g$ is the complete \ke metric on $D$ with Ricci curvature equal to $\lambda=\lambda_0/k$.
Recall the domain $\Omega$ defined in \eqref{ball bundle for domain case}, which now becomes
	\begin{equation*}
		\Omega=\bigl\{(z, \xi)\in D\times \mathbb{C}^k: \bigl(\frac{1}{1-|z|^2}\bigr)^{1/p} |\xi|^2<1 \bigr\}.
	\end{equation*}
	Clearly, we have $\Omega=E_p$. For $m=n+k$, by Example \ref{ex bounded domain of holomorphy}, the Cheng\textendash Yau\textendash Mok solution for domain $\Omega$ is given by
	\begin{equation*}
		u=k^{\frac{n}{m+1}}(GH)^{-\frac{1}{m+1}}\phi(|\xi|h^{\frac{1}{2}}),
	\end{equation*}
	where $G=\det(g_{i\bar{j}})$ and $H=h^k$. A straightforward computation yields $G=\frac{k^n}{p^n(1-|z|^2)^{n+1}}$, and thus
	\begin{equation}\label{CYM solution for egg domain}
		u=p^{\frac{n}{m+1}} (1-|z|^2)^{(n+1+\frac{k}{p})/(m+1)} \phi\bigl(|\xi|(1-|z|^2)^{-\frac{1}{2p}}\bigr).
	\end{equation}
	
	On the other hand, the Bergman kernel $K$ of $E_p$ was computed by D'Angelo \cite{DA94}:
	\begin{equation}\label{Bergman kernel for egg domain}
		K\bigl((z,\xi), (z, \xi)\bigr)=\sum_{i=0}^{n+1} c_i\frac{(1-|z|^2)^{-(n+1)+\frac{i}{p}}}{\big((1-|z|^2)^{\frac{1}{p}}-|\xi|^2 \big)^{k+i}},
	\end{equation}
	where $c_i$ are constants depending on $i, n, k$ and $p$.
	
To establish Proposition \ref{prop egg domain}, 
we assume the Bergman metric $g_{\mathrm{B}}$ of $E_p$ is K\"ahler-Einstein and first follow the work of Fu\textendash Wong \cite{FuWo} to compute the volume form of $g_{\mathrm{B}}$.  Note that a generic boundary point of $E_p$ is smooth and  strictly pseudoconvex (indeed spherical). Take an arbitrary strictly pseudoconvex boundary point $(z_0, \xi_0)$. By using Fefferman's expansion for the Bergman kernel near $(z_0, \xi_0)$ and the argument in Cheng--Yau (\cite{CheYau}, page 510), we deduce that the Ricci curvature of $g_{\mathrm{B}}$ at $(z, \xi) \in E_p$ tends to $-1$ as $(z, \xi)$ approaches $(z_0, \xi_0)$. Thus the K\"ahler-Einstein assumption implies that the Ricci curvature of $g_{\mathrm{B}}$ is equal to $-1$. Then by Proposition 1.2 in \cite{FuWo}, the determinant of $g_{\mathrm{B}}$ equals the Bergman kernel up to a positive constant multiple. On the other hand, the volume form of two  complete \ke metrics of negative Ricci curvature, $\det\bigl( (-\log u)_{s\bar{t}} \bigr)$ and the determinant of $g_{\mathrm{B}}$,  on $E_p$ can only differ by a positive constant multiple.
	As a result, we have  $u^{-(m+1)}=c K$ for some constant $c >0$. Combining this with \eqref{CYM solution for egg domain} and \eqref{Bergman kernel for egg domain}, we obtain
	\begin{equation*}
		p^{-n} (1-|z|^2)^{-(n+1+\frac{k}{p})} \phi^{-(m+1)}\bigl(|\xi|(1-|z|^2)^{-\frac{1}{2p}}\bigr)=c\sum_{i=0}^{m} c_i\frac{(1-|z|^2)^{-(n+1)+\frac{i}{p}}}{\big((1-|z|^2)^{\frac{1}{p}}-|\xi|^2 \big)^{k+i}}.
	\end{equation*}
	After simplification, this becomes
	\begin{equation*}
		\phi^{-(m+1)}\bigl(|\xi|(1-|z|^2)^{-\frac{1}{2p}}\bigr)=c\,p^n \sum_{i=0}^{m} c_i\frac{1}{\big(1-|\xi|^2(1-|z|^2)^{-\frac{1}{p}} \big)^{k+i}}.
	\end{equation*}
	By setting $r=|\xi|(1-|z|^2)^{-\frac{1}{2p}}$, we observe $\phi^{-(m+1)}(r)$ is equal to $c\,p^n \sum_{i=0}^{m} c_i\big(1-r^2 \big)^{-(k+i)}$. Thus, $\phi^{m+1}$ is rational. By Proposition \ref{prop equivalent conditions on rationality}, we get $\lambda=-\frac{n+1}{k}$. Recall that $\lambda=\frac{\lambda_0}{k}=-\frac{p(n+1)}{k}$. So it follows that $p=1$ and the proof is completed.
\end{proof}

\bibliographystyle{plain}
\bibliography{references}

\end{document}